\documentclass[12pt]{amsart}
\usepackage[T2A]{fontenc}
\usepackage[english]{babel}
\usepackage{amsaddr}
\usepackage{amsmath}
\usepackage{amssymb}
\usepackage{amsfonts}
\usepackage{epsfig}
\usepackage{srcltx}
\usepackage{subfigure}
\usepackage{cite}
\usepackage{float}
\usepackage{mathtools}
\usepackage[a4paper,  mag=1000, includefoot,  left=2cm, right=2cm, top=2cm, bottom=2cm, headsep=1cm, footskip=1cm ]{geometry}

\newtheorem{Th}{Theorem}

\newtheorem{Cor}{Corollary}
\newtheorem{Def}{Definition}
\begin{document}
\thispagestyle{empty}

\title[Damped perturbations of systems with centre-saddle bifurcation]
{Damped perturbations of systems with centre-saddle bifurcation}
\author{Oskar A. Sultanov}

\address{
Institute of Mathematics, Ufa Federal Research Centre, Russian Academy of Sciences,
\newline
112, Chernyshevsky street, Ufa, Russia, 450008}
\email{oasultanov@gmail.com}


\maketitle
{\small

{\small
\begin{quote}
\noindent{\bf Abstract.}
An autonomous system of ordinary differential equations in the plane with a centre-saddle bifurcation is considered. The influence of time damped perturbations with power-law asymptotics is investigated. The particular solutions tending at infinity to the fixed points of the limiting system are considered. The stability of these solutions is analyzed when the bifurcation parameter of the unperturbed system takes critical and non-critical values. Conditions that ensure the persistence of the bifurcation in the perturbed system are described. When the bifurcation is broken, a pair of solutions tending to a degenerate fixed point of the limiting system appears in the critical case. It is shown that, depending on the structure and the parameters of the perturbations, one of these solutions can be stable, metastable or unstable, while the other solution is always unstable.
\medskip


\end{quote}
}

\section{Introduction}

In this paper, the effect of time-dependent perturbations on autonomous systems with a centre-saddle bifurcation is investigated. A class of perturbations described by the functions vanishing at infinity in time is considered. We investigate the behaviour of solutions to such asymptotically autonomous systems in the vicinity of a bifurcation point of the corresponding limiting systems. It is well known that in some cases the trajectories of perturbed and unperturbed systems have the same long term behaviour (see, for example,~\cite{LM56,HRT92}). In the general case, the qualitative and asymptotic properties of solutions depend on the limiting system and on the structure of perturbations~\cite{HRT94}. From~\cite{LRS02,KS05,MR08} it follows that damped perturbations can preserve or destroy autonomous bifurcations. See also~\cite{CP10}, where bifurcation phenomena are discussed for more general non-autonomous systems. In this paper, we study the conditions that guarantee the persistence of the bifurcation in the perturbed system, and describe possible asymptotic regimes for solutions at infinity in time when the bifurcation is broken.

The paper is organized as follows. In section~\ref{sec1}, the formulation of the problem is given and the class of perturbations decaying at infinity is described. In section~\ref{sec2}, we construct the asymptotics for particular solutions associated with the fixed points of the corresponding limiting system. The stability of the particular solutions is discussed in section~\ref{sec3}. The instability of some solutions can be justified by linearization. For other solutions, the linear analysis fails and the stability is investigated by constructing appropriate Lyapunov functions. In section~\ref{sec4}, the proposed theory is applied to examples of asymptotically autonomous systems. The paper concludes with a brief discussion of the results obtained.

\section{Problem statement}
\label{sec1}
Consider a non-autonomous system of two differential equations:
\begin{gather}
\label{FulSys}
\frac{dx}{dt}=\partial_y H(x,y;\lambda)+F(x,y,t), \quad \frac{dy}{dt}=-\partial_x H(x,y;\lambda)+G(x,y,t),
\end{gather}
were $H(x,y;\lambda)$, $F(x,y,t)$ and $G(x,y,t)$ are smooth functions defined for all $(x,y)$ in $\mathbb R^2$, $t>0$, and $\lambda\in\mathbb R$.
It is assumed that $G(x,y,t)\to 0$ and $F(x,y,t)\to 0$ as $t\to\infty$ for $(x,y)$ in any compact subset $D\subset \mathbb R^2$, and the limiting Hamiltonian system has a centre-saddle bifurcation such that under variation of the parameter $\lambda$ equilibria of centre and saddle type coalesce and disappear~\cite{HH07}. Without loss of generality, it is assumed that
\begin{gather*}
    H(x,y;\lambda)=\frac{y^2}{2}+V(x;\lambda), \quad \partial_x V(x;\lambda)=(x^2-\lambda)w(x),
\end{gather*}
where $w(x)>0$ for all $(x,\lambda)\in\{|x|<\sqrt{\lambda}+d_0,\lambda\geq 0\}\cup\{|x|<d_0,\lambda< 0\}$ with some $d_0>0$. The paper investigates the influence of $F(x,y,t)$ and $G(x,y,t)$ on a global behaviour of solutions. In particular, we study possible asymptotic regimes in perturbed system \eqref{FulSys} and their stability for various values of the parameter $\lambda$.

It is assumed that the perturbations are described by functions with power-law asymptotics:
\begin{gather}
\label{FG}
F(x,y,t)=\sum_{k=1}^\infty t^{-\frac kq} F_k(x,y), \quad G(x,y,t)=\sum_{k=1}^\infty t^{-\frac kq} G_k(x,y),\quad t\to\infty, \quad q\in \mathbb Z_+.
\end{gather}
Such perturbations appear, for example, in the study of solutions of Painlev\'{e} equations and their disturbances~\cite{IKNF06,BG08}, in the asymptotic analysis of resonance phenomena in nonlinear systems~\cite{LK08,LK14,OS20,OS21}, and in many other problems related to nonlinear non-autonomous systems~\cite{BD79,BAD04,VB07,KF13,ML14}.

Note that the rational powers $k/q$ with $q>1$ in \eqref{FG} can be reduced to the integer exponents $k$ by the change of the independent variable: $\theta=t^{1/q}$. However, in this case, the damped factor $\theta^{-(q-1)}$ appears in the left-hand side of the system, which makes the problem of a long-term behaviour singularly perturbed~\cite{FV05}. In this way a global behaviour of solutions cannot be derived from a corresponding limiting equations that appear if we take $\theta=\infty$. On the other hand, it can easily be checked that system \eqref{FulSys}, written in the variable $\tau=\epsilon t$ with a small parameter $0<\epsilon\ll1$, is singularly perturbed as $\epsilon\to 0$. In some cases the asymptotic solution of such a problem as $\epsilon\to0$ and $0<\tau\leq \mathcal O(1)$ gives a long-term approximation for solutions in the original variable $t=\tau/\epsilon$. However, this approach is usually not used when investigating the behaviour of solutions at infinity~\cite{WW66,FO74,IKNF06,BG08}. Moreover, known asymptotic constructions with a small parameter turn out to be inapplicable in a general case~\cite{LK09}.

Consider a simple example demonstrating possible effects of damped perturbations on a system with a centre-saddle bifurcation:
\begin{gather}
\label{Example}
\frac{d^2x}{dt^2}+x^2-\lambda=t^{-\kappa} \Big(B \frac{dx}{dt}+C\Big), \quad B,C,\lambda \in\mathbb R, \quad \kappa>0.
\end{gather}
This equation in the variables $x,y= \dot x$ takes form \eqref{FulSys} with  $V(x;\lambda)=x^3/3-\lambda x$, $w(x)\equiv 1$, $F(x,y,t)\equiv 0$, and $G(x,y,t)\equiv t^{-\kappa} (B y+C)$. It is easy to see that if $\lambda>0$, the limiting equation has two fixed points: $z_s=(-\sqrt\lambda,0)$ is a saddle and $z_c=(\sqrt\lambda,0)$ is a centre. If $\lambda=0$, the unperturbed equation has a degenerate unstable fixed point $z_0=(0,0)$, which disappears when $\lambda<0$. If $\lambda<0$, all trajectories of the limiting system are unbounded.  Outside the bifurcation value, damped perturbations change the behaviour of solutions in the vicinity of the centre: depending on the perturbation parameters, the trajectories are attracted to $z_c$, are repelled by $z_c$, or remain in some neighbourhood of $z_c$ without attraction to it (see Fig.~\ref{f1}). In this case, the behaviour of solutions near the saddle $z_s$ changes insignificantly.  When $\lambda=0$, the following regimes are possible near the fixed point $z_0$: the trajectories tend to $z_0$ at infinity in time (see~Fig.~\ref{f2}, a); the trajectories oscillate near $z_0$ for a finite time interval and eventually leave its vicinity (see~Fig.~\ref{f2}, b); and the trajectories leave a neighbourhood of $z_0$ without a delay (see~Fig.~\ref{f2}, c).
\begin{figure}
\centering
\subfigure[$B=0.1$ ]{\includegraphics[width=0.3\linewidth]{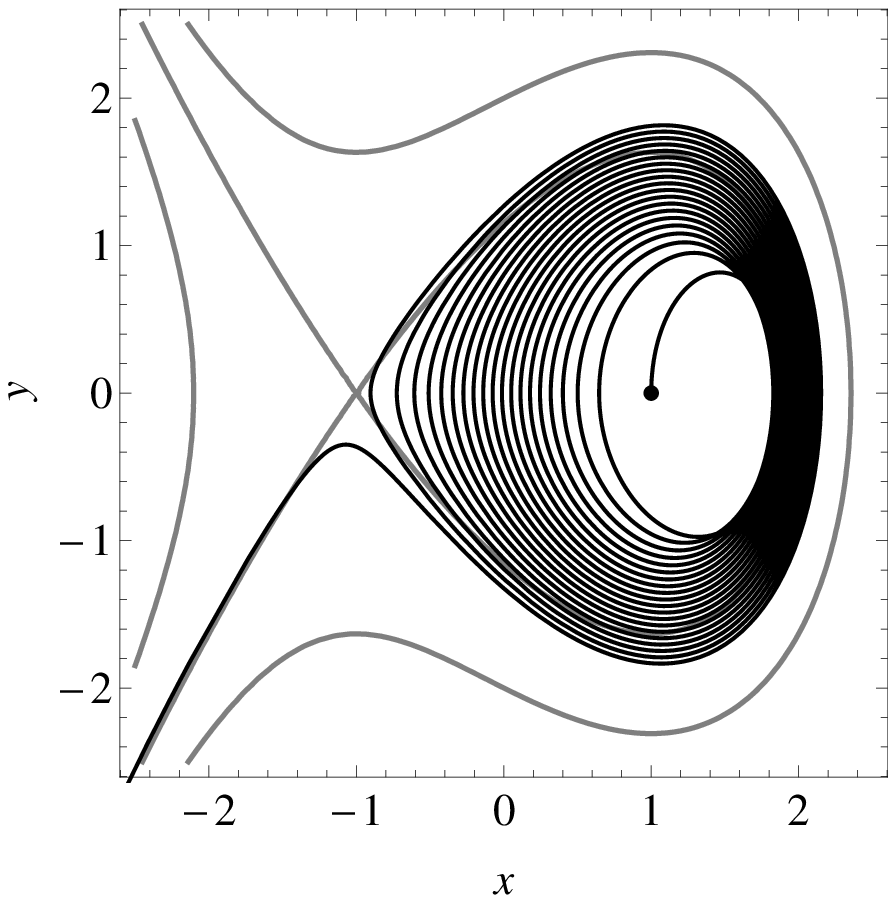}}
\subfigure[$B=0$]{\includegraphics[width=0.3\linewidth]{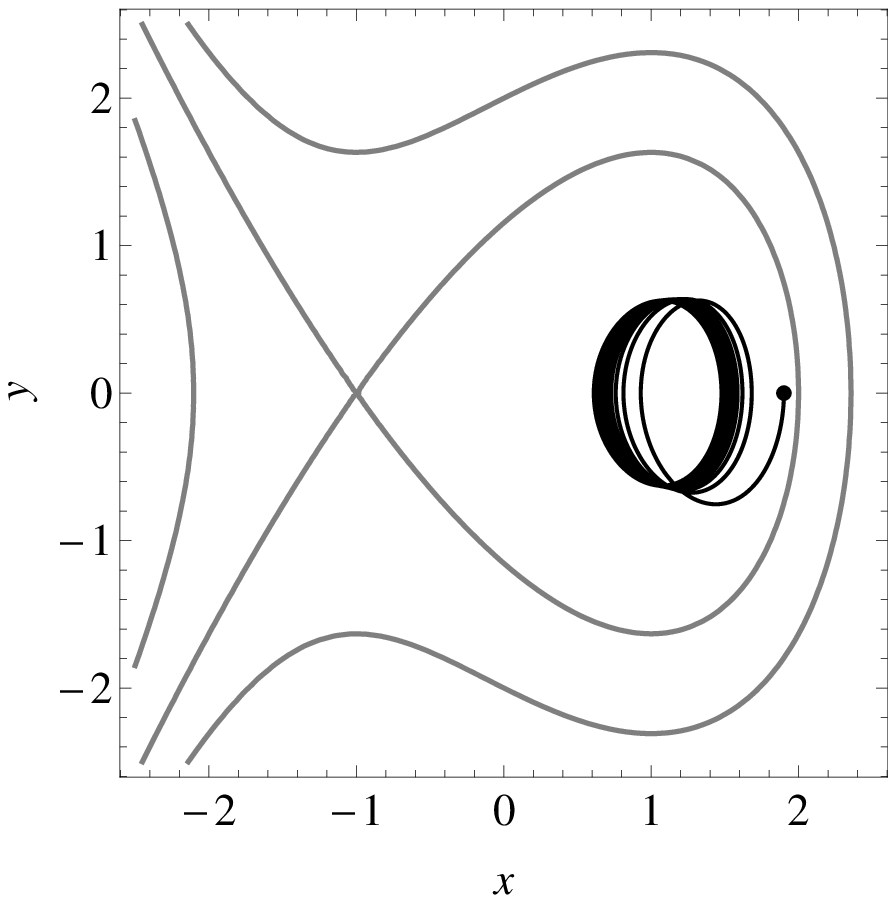}}
\subfigure[$B=-0.5$]{\includegraphics[width=0.3\linewidth]{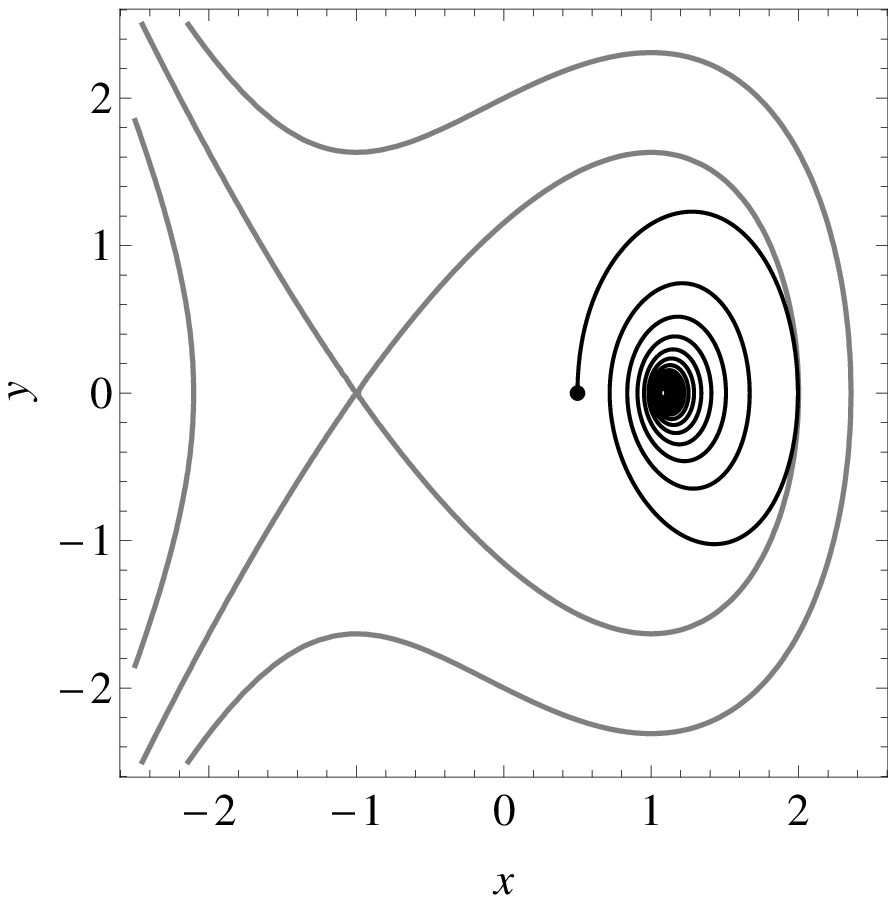}}
\caption{\small The evolution of $(x(t),y(t))$ for solutions of \eqref{Example} with $C=1.5$, $\kappa=0.5$ and $\lambda=1$. The black points correspond to initial data $(x(1),y(1))$. The gray curves correspond to level lines of $H(x,y;1)$.} \label{f1}
\end{figure}

\begin{figure}
\centering
\subfigure[$B=-1.5$ ]{\includegraphics[width=0.3\linewidth]{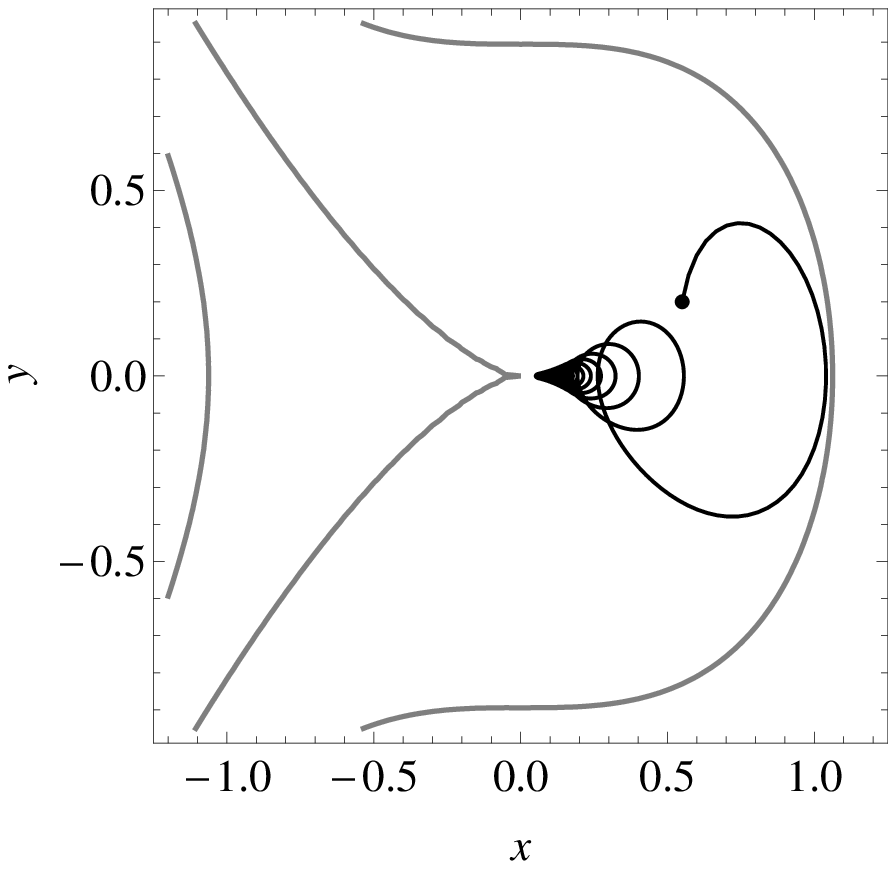}}
\subfigure[ $B=-0.6$]{\includegraphics[width=0.3\linewidth]{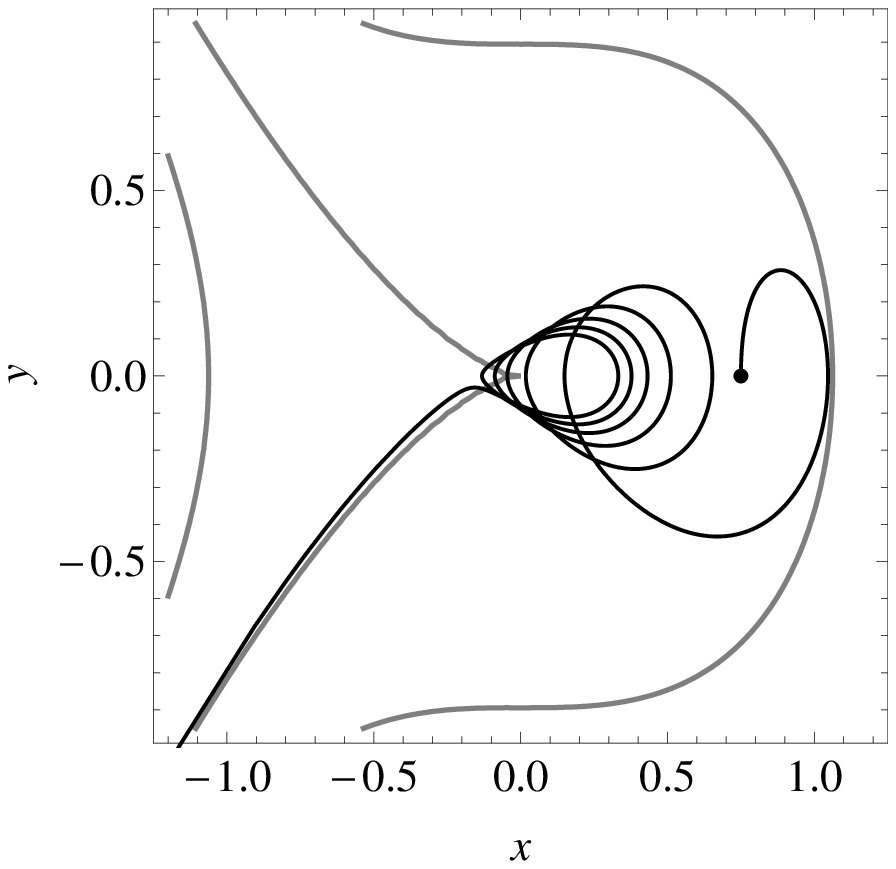}}
\subfigure[$B=0.5$]{\includegraphics[width=0.3\linewidth]{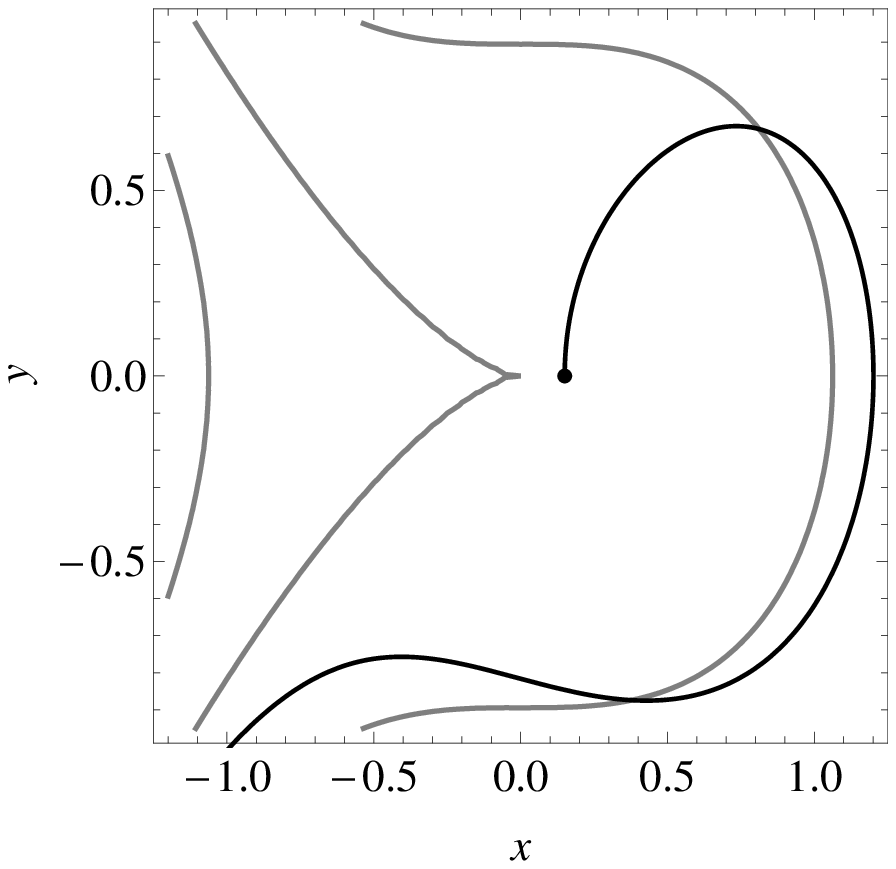}}
\caption{\small The evolution of $(x(t),y(t))$ for solutions of \eqref{Example} with $\kappa=1$, $C=1.5$ and $\lambda=0$. The black points correspond to initial data $(x(1),y(1))$. The gray curves correspond to level lines of $H(x,y;0)$.} \label{f2}
\end{figure}

\section{Particular solutions}
\label{sec2}
Consider particular solutions of perturbed system \eqref{FulSys} tending to the fixed points of the limiting Hamiltonian system.
The simplest asymptotic expansion for such solutions can be constructed in the form of power series with constant coefficients:
\begin{gather}
\label{AS}
x(t)=\sigma+\sum_{k=1}^{\infty} t^{-\frac kq} x_k, \quad y(t)=\sum_{k=1}^{\infty} t^{-\frac kq} y_k, \quad t\to\infty.
\end{gather}
Substituting these series in system \eqref{FulSys} and grouping the terms of the same power of $t$ yield $\sigma^2=\lambda$ and
the following chain of linear equations for the coefficients $x_k$, $y_k$ as $k\geq 1$:
\begin{gather}
    \label{sys}
        \begin{split}
            -2\sigma  w(\sigma) x_k +G_k(\sigma,0)=g_k,\quad
            y_k+F_k(\sigma,0)=f_k,
        \end{split}
\end{gather}
where
$g_1= f_1= 0$, and $g_k$, $f_k$ as $k\geq 2$ are expressed through $x_1,y_1,\dots, x_{k-1},y_{k-1}$:
\begin{eqnarray*}
g_k&=& -  \sum_{
        \substack{
            l+\alpha_1+\ldots+i\alpha_i+\beta_1+\ldots+j\beta_j=k\\
           \alpha_1+\dots+\alpha_i+\beta_1+\dots+\beta_j\geq 1
                }
            }  x_1^{\alpha_1}\cdots x_i^{\alpha_i}y_1^{\beta_1}\cdots y_j^{\beta_j}\frac{ \partial^{\alpha_1+\dots+\alpha_i}_x\partial_y^{\beta_1+\dots+\beta_j} G_l(\sigma,0) }{(\alpha_1+\dots+\alpha_i+\beta_1+\dots+\beta_j)!}- \frac{k-q}{q} y_{k-q}\\
            && - \sum_{
        \substack{
           \alpha_1+\ldots+i\alpha_i=k\\
         \alpha_1+\dots+\alpha_i\geq 2
                }
           }
            \frac{ x_1^{\alpha_1}\cdots x_i^{\alpha_i} }{(\alpha_1+\dots+\alpha_i)!}\partial^{1+\alpha_1+\dots+\alpha_i}_x  V(\sigma),\\
f_k&=& -  \sum_{
        \substack{
            l+\alpha_1+\ldots+i\alpha_i+\beta_1+\ldots+j\beta_j=k\\
           \alpha_1+\dots+\alpha_i+\beta_1+\dots+\beta_j\geq 1
                }
            }  x_1^{\alpha_1}\cdots x_i^{\alpha_i}y_1^{\beta_1}\cdots y_j^{\beta_j}\frac{ \partial^{\alpha_1+\dots+\alpha_i}_x\partial_y^{\beta_1+\dots+\beta_j} F_l(\sigma,0)}{(\alpha_1+\dots+\alpha_i+\beta_1+\dots+\beta_j)!} - \frac{k-q}{q} x_{k-q},
\end{eqnarray*}
where $x_i= y_i=0$ for $i\leq 0$. System \eqref{sys} is solvable with $\sigma=\pm\sqrt\lambda$ whenever $\lambda>0$.
It can easily be checked that if \begin{gather}\label{as0}
  \exists\, n\leq q: \quad   F_i(x,y)\equiv 0, \quad G_i(x,y)\equiv 0\quad \text{ for all} \quad i<n,
\end{gather}
then $x_k=y_k=0$ in \eqref{AS} at least for all $1\leq k< n$.

Thus, we have the following.
\begin{Th}
\label{Th1}
If $\lambda>0$, system \eqref{FulSys} has two different solutions $x_\ast(t)$, $y_\ast(t)$ with asymptotic expansion in the form \eqref{AS} with $\sigma=\pm\sqrt\lambda$.
\end{Th}
\begin{proof}
The existence of solutions with power-law asymptotics at infinity in time follows from~\cite{KF13,AK89}.
\end{proof}

If $\lambda=0$, the form of asymptotic solutions depends essentially on the structure of perturbations \eqref{FG}. Together with \eqref{as0}, consider one of the following assumptions:
\begin{align}
\label{as1}
        &\exists\,m<n: \quad  G_{n+m}(0,0)\neq 0, \quad G_{n+l}(0,0)= 0  \quad \forall\, l<m;\\
\label{as2}
         & G_{n+l}(0,0)= 0  \quad \forall\, l< n, \quad G_{2n}(0,0)-F_n(0,0)\big(\partial_y G_{n}(0,0)+\delta_{n,q}\big)\neq 0,
\end{align}
where $\delta_{n,q}$ is the Kronecker delta.

{\bf 1}. Let assumptions \eqref{as0} and \eqref{as1} hold. Then the asymptotic solution is constructed in the form:
\begin{gather}
\label{AS1}
x(t)=\sum_{k=n+m}^{\infty} t^{-\frac{k}{2q}} x_k, \quad y(t)=\sum_{k=2n}^{\infty} t^{-\frac{k}{2q}} y_k,\quad t\to\infty,
\end{gather}
where
\begin{gather}
    \label{eq1}
        -x_{n+m}^2 w(0)+G_{n+m}(0,0)=0,\quad y_{2n}+F_n(0,0)=0.
\end{gather}
Note that the asymptotic solution in the form \eqref{AS1} does not exist when $G_{n+m}(0,0)< 0$. If $G_{n+m}(0,0)>0$, system \eqref{eq1} has two different roots: $x_{n+m}=\pm\mu$, $y_{2n}=-F_n(0,0)$, where
\begin{gather}
\label{mueq}
    \mu=\sqrt{\frac{G_{n+m}(0,0)}{w(0)}}.
\end{gather}
The remaining coefficients $x_{n+m+k}$, $y_{2n+k}$  as $k\geq 1$ are determined from the following system of equations:
\begin{gather*}
        \begin{split}
            -2  x_{n+m} w(0) x_{n+m+k} =\tilde g_k,\quad
            y_{2n+k}=\tilde f_k,
        \end{split}
\end{gather*}
where the functions $\tilde g_k$, $\tilde f_k$ are expressed through $x_{n+m},\dots,x_{n+m+k-1},y_{2n},\dots,y_{2n+k-1}$.
If $n=1$ and $m=0$, we have $x_1=\pm\sqrt{G_1(0,0)/w(0)}$, $y_2=-F_1(0,0)$,
\begin{eqnarray*}
\tilde g_1&\equiv& x_{1}^3 \partial_x w(0) -x_{1} \partial_x G_1(0,0),\\
\tilde g_2&\equiv& x_2^2 w(0)+  3x_1^2x_2 \partial_x w(0) +\frac{x_1^4}{2} \partial^2_xw(0)  -G_2(0,0)-x_2\partial_x G_1(0,0)-y_2\partial_yG_1(0,0)-\frac{x_1^2}{2}\partial_x^2G_1(0,0),\\
\tilde f_3&\equiv&-x_1\partial_x F_1(0,0),\\
\tilde f_4&\equiv&-F_2(0,0)-x_2\partial_x F_1(0,0)-y_2\partial_y F_1(0,0)-\frac{x_1^2}{2}\partial^2_x F_1(0,0).
\end{eqnarray*}
 \begin{Th}
 \label{Th2}
 Let $\lambda=0$ and assumptions \eqref{as0}, \eqref{as1} hold. If
 $G_{n+m}(0,0)>0$, then system \eqref{FulSys} has two different solutions $x_\ast(t)$, $y_\ast(t)$ with asymptotic expansion in the form \eqref{AS1} with $x_{n+m}=\pm\mu$.
\end{Th}

If $G_{n+m}(0,0)<0$, the asymptotic solution in the form \eqref{AS1} does not exists. Moreover, in this case, the trajectories of the perturbed system behave like solutions to the limiting system with $\lambda<0$. Indeed, the change of variables
\begin{gather*}
  x(t)=\vartheta_1  a(\tau) t^{-\frac{n+m}{2q}}, \quad  y(t)=\sqrt{\vartheta_1^3 w(0)} b(\tau)  t^{-\frac{3(n+m)}{4q}}, \quad \tau=\vartheta_2 \sqrt{\vartheta_1 w(0)} t^{1-\frac{n+m}{4q}},
\end{gather*}
with $\vartheta_1=(|G_{n+m}(0,0)|/w(0))^{1/2}>0$ and $\vartheta_2={4q}/(4q-n-m)>0$,
transforms \eqref{FulSys} into the following form:
\begin{gather*}
    \frac{da}{d\tau}=b+\mathcal O(\tau^{-\frac{3n-m}{4q-n-m}})+\mathcal O(\tau^{-1}), \quad \frac{db}{d\tau}=-(a^2+1 )+\mathcal O(\tau^{-\frac{2(n-m)}{4q-n-m}})+\mathcal O(\tau^{-1}), \quad \tau\to\infty.
\end{gather*}
It is readily seen that every solution of a corresponding Hamiltonian limiting system with $\tilde H(a,b)= a+a^3/3+b^2/2$ leaves a neighbourhood of the origin (see~Fig.~\ref{fab}, a). Hence, the trajectories of system \eqref{FulSys} with  $\lambda=0$ and initial data from the vicinity of the fixed point $(0,0)$ have the same property (see~Fig.~\ref{fab}, b).

\begin{figure}
\centering
\subfigure[]{\includegraphics[width=0.3\linewidth]{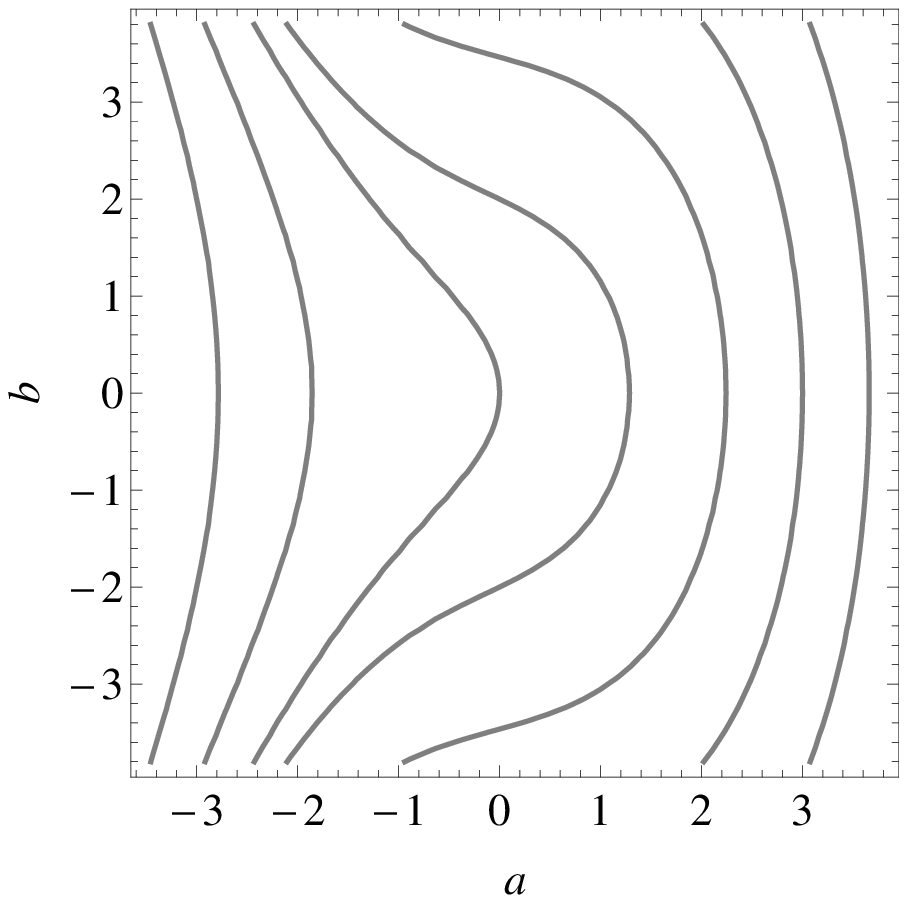}}\quad
\subfigure[]{\includegraphics[width=0.3\linewidth]{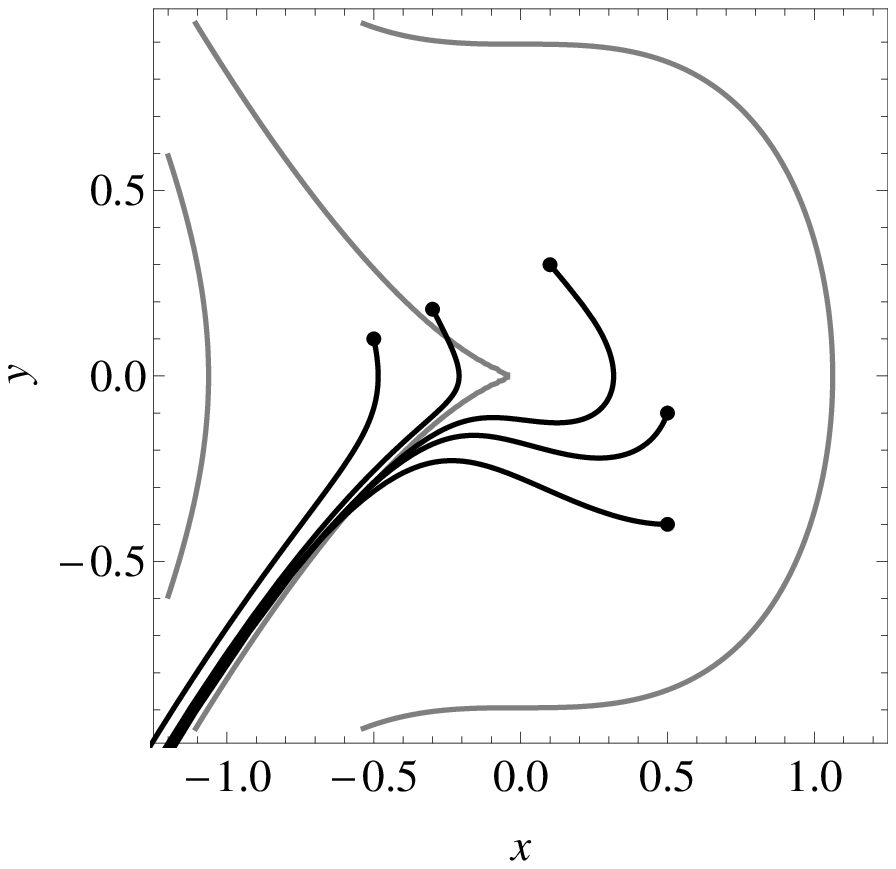}}
\caption{\small (a) The level lines of $\tilde H(a,b)$. (b) The evolution of $(x(t),y(t))$ for solutions of \eqref{Example} with $B=-1.5$, $C=-0.1$, $\kappa=1$ and $\lambda=0$. The black points correspond to initial data $(x(2),y(2))$. The gray curves correspond to level lines of $H(x,y;0)$.} \label{fab}
\end{figure}

{\bf 2}. Let assumptions \eqref{as0} and \eqref{as2} hold. The asymptotic solution is sought in the form:
\begin{gather}
\label{AS2}
x(t)=\sum_{k=n}^{\infty} t^{-\frac{k}{q}} x_k, \quad y(t)=\sum_{k=n}^{\infty} t^{-\frac{k}{q}} y_k,\quad t\to\infty.
\end{gather}
Substituting \eqref{AS2} in system \eqref{FulSys} and grouping the terms of the same power of $t$ yield the following chain of equations:
\begin{gather}
\label{sys2}
         \begin{split}
          &       -x_{n}^2 w(0)+x_{n}\partial_x G_{n}(0,0)+G_{2n}(0,0)+y_{n}\big(\partial_y G_{n}(0,0)+\delta_{n,q}\big)=0,\quad y_{n}+F_n(0,0)=0,\\
         & \Big( -2  x_{n} w(0) + \partial_x G_{n}(0,0) \Big) x_{n+k}+\Big(\partial_y G_{n}(0,0)+\Big(1+\frac{k}{q}\Big)\delta_{n,q}\Big)y_{n+k}=\hat g_k,\quad
            y_{n+k}=\hat f_k,
        \end{split}
\end{gather}
where the functions $\hat g_k$, $\hat f_k$ as $k\geq 1$ are expressed through $x_{n},y_n,\dots,x_{n+k-1},y_{n+k-1}$. For instance,
\begin{eqnarray*}
\hat g_k&\equiv& -G_{2n+k}(0,0)+w(0)\sum_{i=0}^{k-1}x_{n+i}x_{n+k-i}-\sum_{i=0}^{k-1}
  \big(x_{n+i}\partial_x  +y_{n+i}\partial_y \big)G_{n+k-i}(0,0)\\
           &&+\Big(1-\frac{2n+k}{q}\Big) (1-\delta_{n,q}) y_{2n+k-q},\\
\hat f_k&\equiv&-F_{n+k}(0,0)+\Big(1-\frac{n+k}{q}\Big)x_{n+k-q}  \quad \text{for}\ \ k\leq n-1,
\end{eqnarray*}
\begin{eqnarray*}
\hat g_n&\equiv& -G_{3n}(0,0)+w(0)\sum_{i=0}^{n-1} x_{n+i}x_{2n-i}-\sum_{i=0}^{n-1}\big(x_{n+i}\partial_x+y_{n+i}\partial_y  \big) G_{2n-i}(0,0)\\
           &&+x_n^3\partial_x w(0)-\Big(\frac{x_n^2}{2}\partial_x^2+x_n y_n\partial_x\partial_y  +\frac{y_n^2}{2}\partial_y^2  \Big) G_n(0,0) +\Big(1-\frac{3n}{q}\Big) (1-\delta_{n,q})  y_{3n-q},\\
\hat f_n&\equiv&-F_{2n}(0,0)+\Big(1-\frac{2n}{q}\Big)x_{2n-q}-\big(x_n\partial_x+y_n\partial_y \big) F_n(0,0),
\end{eqnarray*}
where it is assumed that $x_i=y_i=0$ for $i< n$.
It can easily be checked that system \eqref{sys2} is solvable whenever
\begin{gather*}
\Delta_n\equiv \big(\partial_x G_n(0,0)\big)^2-4 w(0)\Big( \big(\partial_y G_{n}(0,0)+\delta_{n,q}\big)F_n(0,0)-G_{2n}(0,0)\Big)>0.
\end{gather*}
In this case,  $x_n=\nu_\pm$, $y_n=-F_n(0,0)$, where
\begin{gather}
\label{nueq}
   \nu_\pm=\frac{1}{2w(0)}\Big(\partial_xG_n(0,0)\pm \sqrt{\Delta_n}\Big)\neq 0.
\end{gather}

 \begin{Th}
\label{Th3}
 Let $\lambda=0$ and assumptions \eqref{as0}, \eqref{as2} hold. If $\Delta_n>0$, system \eqref{FulSys} has two different solutions $x_\ast(t)$, $y_\ast(t)$ with asymptotic expansion in the form \eqref{AS2} with $x_{n}=\nu_\pm$.
\end{Th}

Note that if $\Delta_{n}<0$, the asymptotic solution in the form \eqref{AS2} is not constructed. In this case, consider the following change of variables:
\begin{gather}\label{subs}
    \begin{split}
  x(t) &= \Big(\frac{\partial_x G_n(0,0)}{2w(0)}+ \vartheta_1 a(\tau)\Big)t^{-\frac{n}{q}}, \\
  y(t) &=  -F_n(0,0)t^{-\frac{n}{q}}+ \sqrt{\vartheta_1^3w(0)} b(\tau) t^{-\frac{3n}{2q}},\\
  \tau &= \vartheta_2\sqrt{\vartheta_1w(0)} t^{1-\frac{n}{2q}},
      \end{split}
\end{gather}
where $\vartheta_1=\sqrt{|\Delta_n|}/(2w(0))>0$ and $\vartheta_2={2q}/(2q-n)>0$.
Substituting \eqref{subs} in \eqref{FulSys} yields
\begin{gather*}
    \frac{da}{d\tau}=b+\mathcal O(\tau^{-\frac{n}{2q-n}})+\mathcal O(\tau^{-1}), \quad \frac{db}{d\tau}=-(a^2+1)+\mathcal O(\tau^{-\frac{n}{2q-n}})+\mathcal O(\tau^{-1}), \quad \tau\to\infty.
\end{gather*}
Thus, as in the previous case, we see that the qualitative behaviour of perturbed system \eqref{FulSys} is the same as that of the limiting system with $\lambda<0$.

The results of this section are shown in Table~\ref{Table1}.

\begin{table}
\begin{tabular}{l|l|l|l}
\hline \rule{0cm}{0.5cm}
  {\bf Case }               & {\bf Assumptions   }                  & {\bf Asymptotic behaviour} & {\bf Ref.}    \\ [0.1cm]
\hline \rule{0cm}{0.5cm}
    $ \lambda>0$ & \eqref{as0}                      & \eqref{AS}, $\sigma=\pm \sqrt\lambda$ & Th.~\ref{Th1} \\ [0.1cm]
\hline \rule{0cm}{0.5cm}
 $ \lambda=0$ & \eqref{as0}, \eqref{as1}, $G_{n+m}(0,0)>0$   & \eqref{AS1}, $x_{n+m}=\pm\mu$ & Th.~\ref{Th2} \\ [0.1cm]
\cline{2-4}  \rule{0cm}{0.5cm}
  & \eqref{as0}, \eqref{as2}, $\Delta_n>0$     &  \eqref{AS2}, $x_{n}=\nu_\pm$& Th.~\ref{Th3} \\ [0.1cm]
\hline
\end{tabular}
\bigskip
\caption{{\small Particular solutions of system \eqref{FulSys}.}}
\label{Table1}
\end{table}

\section{Stability analysis}
\label{sec3}
\subsection{Linear analysis}
Let $x_\ast(t)$, $y_\ast(t)$ be one of the particular solutions with asymptotics \eqref{AS}, \eqref{AS1} or \eqref{AS2}. The substitution $x(t)=x_\ast(t)+\xi(t)$, $y(t)=y_\ast(t)+\eta(t)$ into \eqref{FulSys} gives the following system with a fixed point at $(0,0)$:
\begin{align}
\label{XES}
\begin{split}
&\frac{d\xi}{dt}=\eta+F(x_\ast+\xi,y_\ast+\eta,t)-F(x_\ast,y_\ast,t), \\
&\frac{d\eta}{dt}=-\partial_x V(x_\ast+\xi;\lambda)+\partial_x V(x_\ast;\lambda)+G(x_\ast+\xi,y_\ast+\eta,t)-G(x_\ast,y_\ast,t).
\end{split}
\end{align}
Consider the linearized system:
\begin{gather*}
    \frac{d}{dt}\begin{pmatrix} \xi \\ \eta \end{pmatrix} = A(t) \begin{pmatrix} \xi  \\ \eta \end{pmatrix}, \ \
    A(t):=
        \begin{pmatrix*}[r]
            \displaystyle   \partial_x F(x_\ast,y_\ast,t) & \displaystyle  1+\partial_y F(x_\ast,y_\ast,t)  \\
            \displaystyle   -\partial^2_xV(x_\ast;\lambda)+\partial_x G(x_\ast,y_\ast,t) & \displaystyle   \partial_y G(x_\ast,y_\ast,t)   \end{pmatrix*}.
\end{gather*}
 The roots of the corresponding characteristic equation $|A(t)-e I|=0$ have the form
\begin{gather}\label{ev}
    e_\pm(t)=\frac{1}{2}\Big( {\hbox{\rm tr}}  A(t) \pm \sqrt{ D(t) }\Big),
\end{gather}
where $D(t)=  ({\hbox{\rm tr}}{A}(t) )^2 -4 {\hbox{\rm det}}  {A}(t)$.

Let $\lambda>0$ and assumption \eqref{as0} hold. Then $x_\ast(t)=\sigma+\mathcal O(t^{-n/q})$, $y_\ast(t)=\mathcal O(t^{-n/q})$ as $t\to\infty$, where $\sigma=\pm\sqrt\lambda$. In this case,
\begin{eqnarray*}
     D(t) = -8 \sigma w(\sigma)+\mathcal O(t^{-\frac nq}), \quad {\hbox{\rm tr}}  A(t) =\mathcal O( t^{-\frac nq}),\quad t\to\infty.
 \end{eqnarray*}
  It can easily be checked that if $\sigma=-\sqrt\lambda$, the eigenvalues $e_+(t)$ and $e_-(t)$ are real of different signs:
\begin{gather*}
e_{\pm}(t)=\pm  \lambda ^{\frac 14}\sqrt{2w(-\sqrt\lambda)}+\mathcal O(t^{-\frac nq}), \quad t\to\infty.
\end{gather*}
This implies that the fixed point $(0,0)$ of system \eqref{XES} is a saddle in the asymptotic limit, and the corresponding
particular solution $x_\ast(t)$, $y_\ast(t)$ with $\sigma=-\sqrt\lambda$ is unstable (see, for example,~\cite{HK02}).
\begin{Th}\label{uns0}
Let $\lambda>0$ and assumption \eqref{as0} hold. Then the solution $x_\ast(t)$, $y_\ast(t)$  with asymptotics \eqref{AS}, $\sigma=-\sqrt\lambda$ is unstable.
\end{Th}

Now let $\lambda=0$. Consider the particular solutions $x_\ast(t)$, $y_\ast(t)$ having asymptotics \eqref{AS1} with $x_{n+m}=\pm\mu$, where $\mu>0$ is defined by \eqref{mueq}. Then
\begin{eqnarray*}
         D(t)= -8 w(0) x_{n+m} t^{-\frac{n+m}{2q}} \big(1+\mathcal O(t^{-\frac{1}{2q}})\big),\quad {\hbox{\rm tr}}  A(t)=\mathcal O( t^{-\frac nq}),\quad t\to\infty.
 \end{eqnarray*}
If $x_{n+m}=-\mu$, then both eigenvalues $e_\pm(t)$ are real of different signs: $e_{\pm}(t)=\pm t^{- (n+m)/{4q}} \sqrt{8w(0)\mu} \big(1+\mathcal O(t^{-{1}/{2q}})\big)$, and the corresponding particular solution $x_\ast(t)$, $y_\ast(t)$ with asymptotics \eqref{AS1} is unstable.
\begin{Th}\label{uns1}
Let $\lambda=0$ and assumptions \eqref{as0}, \eqref{as1} hold. Then the solution $x_\ast(t)$, $y_\ast(t)$ with asymptotics \eqref{AS1}, $x_{n+m}=-\mu$ is unstable.
\end{Th}

Finally, consider the particular solutions $x_\ast(t)$, $y_\ast(t)$ having asymptotics \eqref{AS2} with $x_{n}=\nu_\pm$, where $\nu_-<\nu_+$ are defined by \eqref{nueq}. In this case,
\begin{eqnarray*}
     D(t)= 4\mathcal N_n t^{-\frac nq} \big(1+\mathcal O(t^{-\frac{1}{q}})\big), \quad {\hbox{\rm tr}}  A(t) =\mathcal O( t^{-\frac nq}),\quad t\to\infty,
 \end{eqnarray*}
where $\mathcal N_{n}:=- 2 w(0) x_{n}+\partial_x G_n(0,0)$. If $\mathcal N_{n}>0$, then both eigenvalues \eqref{ev} are real of different signs: $e_{\pm}(t)=\pm t^{- n/{2q}} \sqrt{4\mathcal N_{n}} \big(1+\mathcal O(t^{-{1}/{q}})\big)$, and the corresponding particular solution $x_\ast(t)$, $y_\ast(t)$ with asymptotics \eqref{AS2} is unstable.
It can easily be checked that  $\mathcal N_{n}=\sqrt{\Delta_n}>0$ if $x_{n}=\nu_-$; and $\mathcal N_{n}=-\sqrt{\Delta_n}<0$ if $x_{n}=\nu_+$.
Thus, we have the following:

\begin{Th}\label{uns2}
Let $\lambda=0$ and assumptions \eqref{as0}, \eqref{as2} hold.
Then the solution $x_\ast(t)$, $y_\ast(t)$ with asymptotics \eqref{AS2}, $x_{n}=\nu_-$ is unstable.
\end{Th}

Consider the following cases that are not covered by Theorems~\ref{uns0}, \ref{uns1}, and \ref{uns2}:
\begin{description}
  \item[Case I] $\lambda>0$ and $x_\ast(t)=\sqrt\lambda+\mathcal O(t^{-n/q})$ as $t\to\infty$;
  \item[Case II] $\lambda=0$ and $x_\ast(t)=\mu t^{-(n+m)/2q}(1+\mathcal O(t^{-1/2q}))$ as $t\to\infty$;
  \item[Case III] $\lambda=0$ and $x_\ast(t)=\nu_+ t^{-n/q}(1+\mathcal O(t^{-1/q}))$ as $t\to\infty$.
\end{description}
In these cases, the roots of the characteristic equation are complex:
\begin{gather*}
    \begin{array}{ll}
     e_{\pm}(t)=\pm i  \lambda^{\frac 14} \sqrt{2w\big(\sqrt\lambda\big)}+\mathcal O(t^{-\frac nq}),&  \text{ in {\bf Case I}},\\
    e_{\pm}(t)=\pm i t^{- \frac{n+m}{4q}} \sqrt{8w(0)\mu} \big(1+\mathcal O(t^{-\frac{1}{2q}})\big),  &  \text{ in {\bf Case II}},\\
   e_{\pm}(t)=\pm i t^{- \frac{n}{2q}} \sqrt[4]{16 \Delta_n}\big(1+\mathcal O(t^{-\frac{1}{q}})\big),  &  \text{ in {\bf Case III}}
\end{array}
\end{gather*}
and $\Re e_\pm(t)=\mathcal O(t^{-n/q})$ as $t\to\infty$. Hence, the fixed point $(0,0)$ of system \eqref{XES} is a centre in the asymptotic limit, and the linear stability analysis fails (see, for example,~\cite{OS20a}).

\subsection{Nonlinear analysis}
In this subsection we use Lyapunov function method to investigate the cases where linear analysis does not work. This method is based on the construction of appropriate positive definite functions with a sign definite derivative along the trajectories of the system.

{\bf 1}. Consider first {\bf Case I} when $\lambda>0$.  Assume that
 \begin{gather}
   \label{as00}
   \exists\, h\geq 0: \quad \partial_x F_k(x,y)\equiv -\partial_y G_k(x,y)\quad  \forall\, k<n+h;
 \end{gather}
and define $\gamma_{n+h}(\lambda)=\partial_x F_{n+h}(\sqrt\lambda,0)+\partial_y G_{n+h}(\sqrt\lambda,0)$.

\begin{Th}\label{st0}
Let $\lambda>0$ and assumptions \eqref{as0}, \eqref{as00}  hold.
\begin{itemize}
  \item If $\gamma_{n+h}(\lambda)<0$, then the solution $x_\ast(t)$, $y_\ast(t)$  with asymptotics \eqref{AS}, $\sigma=\sqrt\lambda$ is stable.
  \item If $\gamma_{n+h}(\lambda)>0$, then the solution $x_\ast(t)$, $y_\ast(t)$  with asymptotics \eqref{AS}, $\sigma=\sqrt\lambda$ is unstable.
\end{itemize}
\end{Th}
 \begin{proof}
Note that system \eqref{XES} can be written in the following form:
\begin{gather}\label{XEH}
\frac{d\xi}{dt}=\partial_\eta \mathcal H(\xi,\eta,t), \quad \frac{d\eta}{dt}=-\partial_\xi  \mathcal  H(\xi,\eta,t)+\mathcal P(\xi,\eta,t),
\end{gather}
where
\begin{eqnarray*}
   \mathcal  H(\xi,\eta,t)&=&V(x_\ast+\xi;\lambda)-V(x_\ast;\lambda)-\xi \partial_x V(x_\ast;\lambda)+\frac{\eta^2}{2}\\
                    &&+\int\limits_0^\eta \Big(F(x_\ast+\xi,y_\ast+\theta,t)-  F(x_\ast,y_\ast,t)\Big)\,d\theta-\int\limits_0^\xi \Big(G(x_\ast+\zeta,y_\ast,t)-  G(x_\ast,y_\ast,t)\Big)\,d\zeta,\\
  \mathcal P(\xi,\eta,t)&=&\int\limits_0^\eta\Big( \partial_x F(x_\ast+\xi,y_\ast+\theta,t)+\partial_y G(x_\ast+\xi,y_\ast+\theta,t) \Big)\,d\theta.
 \end{eqnarray*}
Since $x_\ast(t)$, $y_\ast(t)$ has asymptotics \eqref{AS} with $\sigma=\sqrt\lambda$, we have
 \begin{gather*}
 \mathcal  H(\xi,\eta,t)= \mathcal  H_0(\xi,\eta)+\sum_{k=n}^\infty t^{-\frac kq} \mathcal  H_k(\xi,\eta), \quad \mathcal P(\xi,\eta,t)=\sum_{k=n+h}^\infty t^{-\frac kq} \mathcal P_k(\xi,\eta),
\end{gather*}
 where
\begin{eqnarray*}
  \mathcal  H_0(\xi,\eta)&=&V(\sqrt\lambda+\xi;\lambda)-V(\sqrt\lambda;\lambda)-\xi \partial_x V(\sqrt\lambda;\lambda)+\frac{\eta^2}{2},\\
  \mathcal  H_k(\xi,\eta)&=&x_k\Big(\partial_x V(\sqrt\lambda+\xi;\lambda)-\partial_x V(\sqrt\lambda;\lambda)-\xi \partial_x^2 V(\sqrt\lambda;\lambda)\Big)\\&&+\int\limits_0^\eta \Big(F_k(\sqrt\lambda+\xi,\theta)-  F_k(\sqrt\lambda,0)\Big)\,d\theta-\int\limits_0^\xi \Big(G_k(\sqrt\lambda+\zeta,0)-  G_k(\sqrt\lambda,0)\Big)\,d\zeta,\\
  \mathcal P_k(\xi,\eta)&=&\int\limits_0^\eta\Big( \partial_x F_k(\sqrt\lambda+\xi,\theta) +\partial_y G_k(\sqrt\lambda+\xi,\theta)\Big)\,d\theta, \quad n+h\leq k<2n,\\
  \mathcal P_{2n}(\xi,\eta)&=&\int\limits_0^\eta\Big( \partial_x F_{2n}(\sqrt\lambda+\xi,\theta) +\partial_y G_{2n}(\sqrt\lambda+\xi,\theta)\Big)\,d\theta +x_n\partial_\xi    \mathcal P_{n}(\xi,\eta)+y_n \int\limits_0^\eta  \partial_\eta^2   \mathcal P_{n}(\xi,\theta)\, d\theta,
\end{eqnarray*}
etc. It is readily seen that
\begin{gather*}
     \mathcal  H_0(\xi,\eta)=\sqrt\lambda w\big(\sqrt\lambda\big)\xi^2+\frac{\eta^2}{2}+\mathcal O(\rho^3), \quad \mathcal H_k(\xi,\eta)=\mathcal O(\rho^2), \quad \mathcal P_k(\xi,\eta)=\mathcal O(\rho)
\end{gather*}
 as $\rho=\sqrt{\xi^2+\eta^2}\to 0$ for all $k\geq n$.  Taking into account \eqref{as00} and the asymptotic formulas for the particular solution $x_\ast(t)$, $y_\ast(t)$, we obtain $\mathcal P(\xi,\eta,t)=t^{-(n+h)/q}\big(\gamma_{n+h}(\lambda)\eta+\mathcal O(\rho^2)+\mathcal O(\rho t^{-1/q})\big)$ as $\rho\to 0 $ and $t\to\infty$.
 Consider the combination
\begin{gather}
\label{U1}
     U_1(\xi,\eta,t)=\mathcal H(\xi,\eta,t)-\frac{\gamma_{n+h}(\lambda)}{2}t^{-\frac{n+h}{q}}\xi\eta
\end{gather}
as a Lyapunov function candidate for system \eqref{XEH}. It can easily be checked that $U_1(\xi,\eta,t)$ has the following asymptotics:
\begin{gather*}
U_1(\xi,\eta,t) =W_1(\xi,\eta)+\mathcal O(\rho^3)+\mathcal O(\rho^2 t^{-\frac nq}),\quad \rho\to 0,\quad t\to\infty,
\end{gather*}
where $W_1(\xi,\eta)=(\omega_1^2 \xi^2+\eta^2)/2$, $\omega_1^2=2\sqrt\lambda w\big(\sqrt\lambda\big)>0$. Hence, for all $0<\varkappa <1$ there exist $\rho_1>0$ and $t_{1}>0$ such that
\begin{gather*}
 (1-\varkappa ) W_1(\xi,\eta)  \leq U_1(\xi,\eta,t) \leq  (1+\varkappa ) W_1(\xi,\eta)
\end{gather*}
for all $(\xi,\eta,t)\in D_{W_1}(\rho_1,t_{1}):=  \{(\xi,\eta,t)\in\mathbb R^3: W_1(\xi,\eta)\leq \rho_1^2, t\geq t_1\}$.
The total derivative of the function $U_1(\xi,\eta,t)$ with respect to $t$ along the trajectories of system \eqref{XEH} has the form:
\begin{gather*}
\frac{dU_1}{dt}\Big|_\eqref{XEH} =\frac{\partial U_1}{\partial t}+\frac{\partial U_1}{\partial \xi}\partial_\eta \mathcal H+ \frac{\partial U_1}{\partial \eta} \Big(-\partial_\xi\mathcal  H+\mathcal P\Big)=    \gamma_{n+h}(\lambda) t^{-\frac{n+h}{q}} W_1(\xi,\eta) \Big(1+\mathcal O(\rho)+\mathcal O(t^{-\frac 1q})\Big)
\end{gather*}
as $\rho\to 0$ and $ t\to\infty$.
It follows that there exist $0<\rho_0\leq \rho_1$ and $t_0\geq t_1$ such that
\begin{gather}
\label{U1est}
    \begin{split}
&    \frac{dU_1}{dt}\Big|_\eqref{XEH} \leq -   |\gamma_{n+h}(\lambda)| \Big(\frac{1-\varkappa }{1+\varkappa }\Big) t^{-\frac{n+h}{q}} U_1\leq 0 \quad \text{if} \quad \gamma_{n+h}(\lambda)<0,\\
&   \frac{dU_1}{dt}\Big|_\eqref{XEH} \geq   \gamma_{n+h}(\lambda) \Big(\frac{1-\varkappa }{1+\varkappa }\Big) t^{-\frac{n+h}{q}} U_1\geq 0 \quad \text{if} \quad \gamma_{n+h}(\lambda)>0
    \end{split}
\end{gather}
 for all $(\xi,\eta,t)\in D_{W_1}(\rho_0,t_0)$.

 Let $\gamma_{n+h}(\lambda)<0$. Then, for all $0<\varepsilon< \rho_0$ there exists $\delta_\varepsilon := \varepsilon \sqrt{(1-\varkappa)/(1+\varkappa)}/2$  such that
\begin{gather*}
\sup_{W_1(\xi,\eta)\leq \delta_\varepsilon^2} U_1(\xi,\eta,t)\leq (1+\varkappa )  \delta^2_\varepsilon< (1-\varkappa )   \varepsilon^2 \leq \inf_{W_1(\xi,\eta)=\varepsilon^2}U_1(\xi,\eta,t)
\end{gather*}
for all $t>t_0$. The last estimates and the negativity of the total derivative of the function $U_1(\xi,\eta,t)$ ensure that any solution of system \eqref{XEH} with initial data $ W_1(\xi(t_0), \eta(t_0))\leq \delta_\varepsilon^2 $ cannot leave the domain $\{(\xi,\eta)\in\mathbb R^2: W_1(\xi,\eta)\leq \varepsilon^2\}$ as  $t>t_0$. Therefore, the fixed point $(0,0)$ of system \eqref{XES} and the solution $x_\ast(t)$, $y_\ast(t)$ of system \eqref{FulSys} are stable.

Let $\gamma_{n+h}(\lambda)>0$. Consider the solution $\xi(t)$, $\eta(t)$ of system \eqref{XEH} with initial data $\xi(t_0)$, $\eta(t_0)$ such that $W_1(\xi(t_0),\eta(t_0))<\rho_0^2$. Then, by integrating the second inequality in \eqref{U1est} with respect to $t$, we obtain the following:
\begin{gather*}
    \begin{split}
        &   U_1\big(\xi(t),\eta(t),t\big)\geq U_1\big(\xi(t_0),\eta(t_0),t_0\big) \exp\left( \gamma_{n }(\lambda) \Big(\frac{1-\varkappa }{1+\varkappa }\Big)\big (\log t-\log t_0\big)\right),\quad h=0,\\
        &   U_1\big(\xi(t),\eta(t),t\big)\geq U_1\big(\xi(t_0),\eta(t_0),t_0\big) \exp\left(\frac{q\gamma_{n+h}(\lambda)}{h}\Big(\frac{1-\varkappa }{1+\varkappa }\Big)\Big (t_0^{-\frac{h}{q}}-t^{-\frac{h}{q}}\Big)\right),\quad h>0
    \end{split}
\end{gather*}
as $t>t_0$.  Hence, there exists $T_0>t_0$ such that the solution $\xi(t)$, $\eta(t)$ escapes from the domain
$\{(\xi,\eta)\in\mathbb R^2: W_1(\xi, \eta) \leq \rho_0\}$ as $ t > T_0$. This means that the fixed point $(0,0)$ of system \eqref{XES} and the solution $x_\ast(t)$, $y_\ast(t)$ of system \eqref{FulSys} are unstable.
 \end{proof}

 \begin{Cor}\label{Cor1}
If $\gamma_{n+h}(\lambda)<0$ and $n+h\leq q$, the solution $x_\ast(t)$, $y_\ast(t)$  with asymptotics \eqref{AS}, $\sigma=\sqrt\lambda$ is asymptotically stable.
 \end{Cor}
 \begin{proof}
 Indeed, let $\xi(t)$, $\eta(t)$ be a solution of system \eqref{XEH} with initial data $\xi(t_0)$, $\eta(t_0)$ such that $W_1(\xi(t_0),\eta(t_0))<\rho_0$. Integrating the first inequality in \eqref{U1est} with respect to $t$ yields
\begin{gather*}
    \begin{split}
        &   0\leq W_1\big(\xi(t),\eta(t)\big)\leq C_0   \exp\left( - |\gamma_{q }(\lambda)| \Big(\frac{1-\varkappa }{1+\varkappa }\Big) \log t \right),  \quad n+h=q,\\
        &   0\leq W_1\big(\xi(t),\eta(t)\big)\leq C_0 \exp\left(-\frac{q|\gamma_{n+h}(\lambda)|}{q-n-h}\Big(\frac{1-\varkappa }{1+\varkappa }\Big) t^{1-\frac{n+h}{q}}\right), \quad n+h<q
    \end{split}
\end{gather*}
as $t\geq t_0$ with the parameter $C_0>0$, depending on $\rho_0$ and $t_0$. Hence, $\xi(t)\to 0$ and $\eta(t)\to0$ as $t\to\infty$. Returning to the original variables we obtain asymptotic stability of the particular solution $x_\ast(t)$, $y_\ast(t)$.
\end{proof}

Similarly, from \eqref{U1est}, it follows that if $\gamma_{n+h}(\lambda)<0$ and $n+h> q$, the solution $x_\ast(t)$, $y_\ast(t)$  with asymptotics \eqref{AS}, $\sigma=\sqrt\lambda$ is neutrally stable. The same conclusion can be drawn when the perturbations are Hamiltonian:
  \begin{gather}
    \label{as00H}
   \partial_x F(x,y,t)\equiv -\partial_y G(x,y,t) \quad \forall (x,y)\in\mathbb R^2, \quad t>0.
 \end{gather}
Define $d_n=\omega_1^2\partial_y F_n(\sqrt\lambda,0)+G_n(\sqrt\lambda,0)+\partial_x G_n(\sqrt\lambda,0)$, where $\omega_1^2=2\sqrt\lambda w\big(\sqrt\lambda\big)$.

\begin{Th}\label{st0H}
Let $\lambda>0$ and assumptions \eqref{as0}, \eqref{as00H}  hold.
\begin{itemize}
  \item If $d_n>0$, then the solution $x_\ast(t)$, $y_\ast(t)$  with asymptotics \eqref{AS}, $\sigma=\sqrt\lambda$ is stable.
  \item If $d_n<0$, then the solution $x_\ast(t)$, $y_\ast(t)$  with asymptotics \eqref{AS}, $\sigma=\sqrt\lambda$ is unstable.
\end{itemize}
\end{Th}
 \begin{proof}
It follows from \eqref{as00H} that $\mathcal P(\xi,\eta,t)\equiv 0$ and
\begin{gather*}
  \mathcal H_n(\xi,\eta)=\big(G_n(\sqrt\lambda,0)+\partial_xG_n(\sqrt\lambda,0)\big)\frac{\xi^2}{2}+\partial_x F_n(\sqrt\lambda,0)\xi\eta+\partial_y F_n(\sqrt\lambda,0) \frac{\eta^2}{2}+\mathcal O(\rho^3), \quad \rho\to 0,
\end{gather*}
in system \eqref{XEH}. In this case, we use
\begin{gather*}
      U_0(\xi,\eta,t)=\mathcal H(\xi,\eta,t)+t^{-1-\frac{n}{q}}\frac{n}{q}\Big(\big(d_n-2\omega_1^2 \partial_y F_n(\sqrt\lambda,0)\big)\frac{\xi\eta}{4\omega_1^2}+\partial_x F_n(\sqrt\lambda,0) \frac{\xi^2}{2}\Big)
\end{gather*}
as a Lyapunov function candidate for system \eqref{XEH}. The derivative of $U_0(\xi,\eta,t)$ along the trajectories of the system is given by
\begin{gather*}
\frac{dU_0}{dt}\Big|_\eqref{XEH} = -\frac{n d_n}{2q\omega_1^2}t^{-1-\frac{n}{q}} W_1(\xi,\eta) \Big(1+\mathcal O(\rho)+\mathcal O(t^{-\frac 1q})\Big), \quad \rho\to0, \quad t\to\infty.
\end{gather*}
By repeating the arguments used in the proof of Theorem~\ref{st0}, we see that the solution $x_\ast(t)$, $y_\ast(t)$ is stable if $d_n>0$, and unstable if $d_n<0$.
 \end{proof}

{\bf 2}. Now consider {\bf Case II}, when $\lambda=0$ and the particular solution has asymptotics \eqref{AS1} with $x_{n+m}=\mu>0$.
Define \begin{gather*}
\alpha_{n,m}=\gamma_n(0)+ \delta_{n,q}\frac{5(q+m)}{4q},
\end{gather*}
where $\gamma_n(0)=\partial_x F_{n}(0,0)+\partial_y G_{n}(0,0)$.
We have the following:
\begin{Th}\label{st1}
Let $\lambda=0$ and assumptions \eqref{as0}, \eqref{as1} hold.
\begin{itemize}
  \item If $\alpha_{n,m}<0$, then the solution $x_\ast(t)$, $y_\ast(t)$  with asymptotics \eqref{AS1}, $x_{n+m}=\mu$ is stable.
  \item If $\alpha_{n,m}>\delta_{n,q}3(q+m)/(2q)$, then the solution $x_\ast(t)$, $y_\ast(t)$  with asymptotics \eqref{AS1}, $x_{n+m}=\mu$ is unstable.
 \end{itemize}
\end{Th}
 \begin{proof}
It can easily be checked that for the particular solutions with asymptotics \eqref{AS1} the functions $\mathcal H(\xi,\eta,t)$ and $\mathcal P(\xi,\eta,t)$ in system \eqref{XEH} have the following expansions as $t\to\infty$:
 \begin{gather*}
 \mathcal  H(\xi,\eta,t)=\mathcal H_0(\xi,\eta)+\sum_{k=n+m}^\infty t^{-\frac{k}{2q}} \mathcal H_k(\xi,\eta) , \quad \mathcal P(\xi,\eta,t)=\sum_{k=2n}^\infty t^{-\frac {k}{2q}} \mathcal P_k(\xi,\eta),
\end{gather*}
 where
\begin{eqnarray*}
  \mathcal  H_0(\xi,\eta)&=&V(\xi;0)-V(0;0)-\xi\partial_x V(0;0)+\frac{\eta^2}{2},\\
  \mathcal  H_{n+m}(\xi,\eta)&=&x_{n+m}\Big(\partial_x V(\xi;0)-\partial_x V(0;0)-\xi \partial_x^2 V(0;0)\Big),\\
  \mathcal P_{2n}(\xi,\eta)&=&\int\limits_0^\eta\Big( \partial_x F_n(\xi,\theta) +\partial_y G_n(\xi,\theta)\Big)\,d\theta.
\end{eqnarray*}
Furthermore,
\begin{gather} \label{HP}
   \mathcal  H_0(\xi,\eta)=\frac{\eta^2}{2}+  w (0)\frac{\xi^3}{3}+\mathcal O(\rho^4),\quad \mathcal H_{n+m}(\xi,\eta)=  w(0)\mu \xi^2 +\mathcal O(\rho^3),
\end{gather}
as $\rho=\sqrt{\xi^2+\eta^2}\to 0$.
Since the function $\mathcal H(\xi,\eta,t)$ is sign indefinite in the vicinity of the fixed point $(0,0)$, the combination of the form \eqref{U1} can not be used as a Lyapunov function candidate for system \eqref{XEH}.

It can easily be checked that the change of variables
 \begin{gather}\label{XY}
    \xi(t)=t^{-\frac{n+m}{2q}}X(t), \quad \eta(t)=t^{-\frac{3(n+m)}{4q}}Y(t)
 \end{gather}
transform system \eqref{XEH} into
\begin{gather}\label{XYH}
\frac{dX}{dt}=\partial_Y \hat{ \mathcal H}(X,Y,t), \quad \frac{dY}{dt}=-\partial_X \hat{ \mathcal H}(X,Y,t)+\hat{ \mathcal P}(X,Y,t),
\end{gather}
where
\begin{eqnarray*}
    \hat{ \mathcal H}(X,Y,t)&=&t^{\frac{5(n+m)}{4q}}\mathcal H\Big(t^{-\frac{n+m}{2q}}X,t^{-\frac{3(n+m)}{4q}}Y,t\Big)+\frac{n+m}{2q}t^{-1}XY, \\
    \hat{ \mathcal P}(X,Y,t)&=&t^{\frac{3(n+m)}{4q}}\mathcal P\Big(t^{-\frac{n+m}{2q}}X,t^{-\frac{3(n+m)}{4q}}Y,t\Big)+\frac{5(n+m)}{4q} t^{-1} Y.
\end{eqnarray*}
Taking into account \eqref{HP} we get
\begin{eqnarray*}
    \hat{ \mathcal H}(X,Y,t)&=&t^{-\frac{n+m}{4q}}\left( \frac{1}{2}(\omega_2^2 X^2+Y^2)+w(0) \frac{X^3}{3}\right)+\mathcal O\big(R^2t^{-\frac{n+m}{2q}}\big)+\mathcal O (R^2 t^{-1}),\\
    \hat{ \mathcal P}(X,Y,t)&=&t^{-\frac{n}{q}} \big(\alpha_{n,m} Y+\mathcal O(R^2)\big)+\mathcal O\big(R t^{-\frac{n+1}{q}}\big)
\end{eqnarray*}
as $R=\sqrt{X^2+Y^2}\to0$ and $t\to\infty$, where $\omega_2^2=2w(0)\mu >0$.
Consider
\begin{gather*}
     U_2(X,Y,t)=t^{\frac{n+m}{4q}}\hat{\mathcal H}(X,Y,t)- t^{-\frac{3n-m}{4q}}\frac{\alpha_{n,m}}{2}XY
\end{gather*}
as a Lyapunov function candidate for system \eqref{XYH}.
It is clear that $U_2(X,Y,t)$ is locally positive definite:
for all $0<\varkappa <1$ there exist $R_1>0$ and $t_{1}>0$ such that
\begin{gather}
\label{W2ineq}
 (1-\varkappa ) W_2(X,Y)  \leq U_2(X,Y,t) \leq  (1+\varkappa ) W_2(X,Y)
\end{gather}
for all $(X,Y,t)\in D_{W_2}(R_1,t_1):=  \{(X,Y,t)\in\mathbb R^3: W_2(X,Y)\leq R_1^2, t\geq t_1\}$, $W_2(X,Y)=(\omega_2^2 X^2+Y^2)/2$.
The total derivative of $U_2(X,Y,t)$ with respect to $t$ along the trajectories of system \eqref{XYH} has the following asymptotics:
\begin{gather*}
\frac{dU_2}{dt}\Big|_\eqref{XYH} = \alpha_{n,m} t^{-\frac{n}{q}} W_2(X,Y) \Big(1+\mathcal O(R)+\mathcal O(t^{-\frac 1q})+\mathcal O(t^{-\frac{n+m}{4q}})\Big), \quad R\to 0, \quad t\to\infty.
\end{gather*}
Hence, there exist $0<R_0\leq R_1$ and $t_0\geq t_1$ such that
\begin{gather}\label{U2est}
    \begin{split}
        & \frac{dU_2}{dt}\Big|_\eqref{XYH} \leq -   |\alpha_{n,m}| \Big(\frac{1-\varkappa }{1+\varkappa }\Big) t^{-\frac{n}{q}} U_2\leq 0 \quad \text{if} \quad \alpha_{n,m}<0,\\
        & \frac{dU_2}{dt}\Big|_\eqref{XYH} \geq    \alpha_{n,m}\Big(\frac{1-\varkappa }{1+\varkappa }\Big) t^{-\frac{n}{q}} U_2\geq 0 \quad \text{if} \quad \alpha_{n,m}>0
    \end{split}
 \end{gather}
 for all $(X,Y,t)\in D_{W_2}(R_0,t_0)$.

If $\alpha_{n,m}<0$, then from the first inequality in \eqref{U2est} it follows that the fixed point $(0,0)$ of system \eqref{XYH} is stable. Combining this with \eqref{XY}, we get the stability of the particular solution $x_\ast(t)$, $y_\ast(t)$ with $x_{n+m}=\mu$.

Let $\alpha_{n,m}>\delta_{n,q} 3(q+m)/(2q)$. Consider the solution $X(t)$, $Y(t)$  of system \eqref{XYH} with initial data $X(t_0)$, $Y(t_0)$ such that $W_2(X(t_0),Y(t_0))=\delta_0^2<R_0^2$. Integrating the second inequality in \eqref{U2est} with respect to $t$ and taking into account \eqref{W2ineq}, we obtain the following:
\begin{gather*}
    \begin{split}
        &  \omega_2^2\xi^2(t)  +\eta^2(t)\geq 2\delta^2 \Big(\frac{1-\varkappa }{1+\varkappa }\Big)t^{-\frac{3(n+m)}{2q}} \exp\left(\frac{q\alpha_{n,m} }{q-n} \Big(\frac{1-\varkappa }{1+\varkappa }\Big)\Big (t^{1-\frac{n}{q}}-t_0^{1-\frac{n}{q}}\Big)\right), \quad  n<q,\\
        &   \omega_2^2\xi^2(t)  +\eta^2(t) \geq 2\delta^2 \Big(\frac{1-\varkappa }{1+\varkappa }\Big) t^{-\frac{3(n+m)}{2q}} \Big(\frac{t}{t_0}\Big)^{\alpha_{n,m}\big(\frac{1-\varkappa }{1+\varkappa }\big) }, \quad n=q
    \end{split}
\end{gather*}
as $t>t_0$. By choosing $\varkappa\in(0,1)$ small enough, we see that the fixed point $(0,0)$ of system \eqref{XEH} and the particular solution $x_\ast(t)$, $y_\ast(t)$ of \eqref{FulSys} are unstable: for all $\delta_0>0$ there exists $T_0>t_0$  such that $\xi^2(t)+\eta^2(t)\geq \rho_\ast^2$ as $ t > T_0$ with some $\rho_\ast>0$. Hence, the particular solution $x_\ast(t)$, $y_\ast(t)$ with $x_{n+m}=\mu$ is unstable.
\end{proof}

\begin{Cor}
If $\alpha_{n,m}<0$, the solution $x_\ast(t)$, $y_\ast(t)$  with asymptotics \eqref{AS1}, $x_{n+m}=\mu$ is asymptotically stable.
\end{Cor}
\begin{proof}
The justification of asymptotic stability follows from \eqref{XY} and the stability of the fixed point $(0,0)$ in system \eqref{XYH}.
\end{proof}

Note that from \eqref{W2ineq} and the second inequality in \eqref{U2est} it follows that the trivial solution $X(t)\equiv 0$, $Y(t)\equiv 0$ of system \eqref{XYH} is unstable if $n=q$ and $\alpha_{n,m}>0$.
Combining this with \eqref{XY2}, we see that the particular solution $x_\ast(t)$, $y_\ast(t)$ with asymptotics \eqref{AS1}, $x_{n+m}=\mu$ is unstable with weights:
there exists $\varepsilon>0$ such that for all $\delta>0$ there is a solution of system \eqref{FulSys} with initial data $|x(t_0)-x_\ast(t_0)|+|y(t_0)-y_\ast(t_0)|\leq \delta$ that satisfies
\begin{gather*}
t^{\frac{n+m}{2q}}|x(t)-x_\ast(t)| +t^{\frac{3(n+m)}{4q}}|y(t)-y_\ast(t)|>\varepsilon
\end{gather*}
at some $t>t_0$. In this case, it can be shown that the solution is stable on a finite but asymptotically long time interval. Let us specify the definition of stability, which is a variant of the concept of practical stability~\cite{LL61}.
\begin{Def}
The solution $x_\ast(t)$, $y_\ast(t)$ to system \eqref{FulSys} is stable on a finite but asymptotically long time interval if for all $\varepsilon>0$ there exist $\delta_\varepsilon>0$, $K_\varepsilon>0$ such that $K_\varepsilon\to\infty$ as $\varepsilon\to 0$ and any solution $x(t)$, $y(t)$ of system \eqref{FulSys} with initial data  $|x(t_0)-x_\ast(t_0)|+|y(t_0)-y_\ast(t_0)|\leq \delta_\varepsilon$ satisfies the following inequality:
\begin{gather*}
|x(t)-x_\ast(t)|+|y(t)-y_\ast(t)|<\varepsilon
\end{gather*}
for all $1\leq t/t_0\leq K_\varepsilon$.
\end{Def}

\begin{Th}\label{mst1}
Let $\lambda=0$, $n=q$  and assumptions \eqref{as0}, \eqref{as1} hold. If $\alpha_{n,m}<(3q-m)/(4q)$, then the solution $x_\ast(t)$, $y_\ast(t)$  with asymptotics \eqref{AS1}, $x_{n+m}=\mu$ is stable on a finite but asymptotically long time interval.
 \end{Th}
 \begin{proof}
It can easily be checked that the change of variables
  \begin{gather}\label{XY2}
    \xi(t)=t^{-\frac{n+m}{2q}}X(t), \quad \eta(t)=t^{-\frac{m}{q}}Y(t)
 \end{gather}
transforms system \eqref{XEH} into \eqref{XYH} with
\begin{eqnarray*}
     \hat{ \mathcal H}(X,Y,t) & = &t^{\frac{n+3m}{2q}}\mathcal H\Big(t^{-\frac{n+m}{2q}}X,t^{-\frac{m}{q}}Y,t\Big)+\frac{n+m}{2q}t^{-1}XY\\
        & =& t^{\frac{n-m}{2q}}\frac{Y^2}{2}+ t^{-1}\Big( \omega_2^2\frac{X^2}{2}+w(0) \frac{X^3}{3}+\frac{n+m}{2q} XY\Big)+\mathcal O (X^2 t^{-1-\frac{1}{2q}}), \\
    \hat{ \mathcal P}(X,Y,t) & = &t^{\frac{m }{q }}\mathcal P\Big(t^{-\frac{n+m}{2q}}X,t^{-\frac{m}{q}}Y,t\Big)+\frac{n+3m}{2q} t^{-1} Y
            =  \hat\alpha_{n,m} t^{-1}  Y +\mathcal O\big(R t^{-1-\frac{1}{2q}}\big)
 \end{eqnarray*}
as $R=\sqrt{X^2+Y^2}\to0$ and $t\to\infty$, where $\hat\alpha_{n,m}=\alpha_{n,m}-(3q-m)/(4q)<0$, $\omega_2^2=2 w(0) \mu>0$. Hence, there exists $t_\ast>1$ such that for every $t>t_\ast$, the Hamiltonian $\hat{\mathcal H}(X,Y,t)$ is a positive definite quadratic form in the vicinity of the fixed point $(0,0)$. Consider
\begin{gather*}
\hat U_2(X,Y,t):=t^{-\frac{n-m}{2q}}\Big(\hat{\mathcal H}(X,Y,t)-t^{-1}\frac{\hat\alpha_{n,m}}{2}XY\Big),
\end{gather*}
as a Lyapunov function candidate.
It is readily seen that for all $0<\varkappa <1$ there exist $R_1>0$ and $t_1\geq t_\ast$ such that
\begin{gather*}
 (1-\varkappa ) \hat W_2(X,Y,t)  \leq \hat U_2(X,Y,t)\leq (1+\varkappa ) \hat W_2(X,Y,t)
\end{gather*}
for all $(X,Y,t)\in D_{W_2}(R_1,t_1)$, where
\begin{gather*}
    \hat W_2(X,Y,t)=\frac12 \Big( Y^2+t^{-1-\frac{n-m}{2q}}\omega_2^2X^2\Big), \quad W_2(X,Y)\equiv \hat W_2(X,Y,1).
\end{gather*}
The derivative of $\hat U_2(X,Y,t)$ with respect to $t$ along the trajectories of the system has the following asymptotics:
\begin{gather*}
\frac{d\hat U_2}{dt}\Big|_\eqref{XYH} =  \frac{\hat\alpha_{n,m}}{2} t^{-1}   \Big(  Y^2+t^{-1-\frac{n-m}{2q}}\big(\omega_2^2X^2+\mathcal O(R^3)\Big)+\mathcal O(R^2 t^{-1-\frac{1}{2q}} ), \quad R\to0,\quad t\to\infty.
\end{gather*}
Hence, for all $0<\varkappa <1$ there exist $0<R_0\leq R_1$ and $t_0\geq t_1$ such that
\begin{gather}
\label{hU2est}
\frac{d\hat U_2}{dt}\Big|_\eqref{XYH} \leq  -   |\hat\alpha_{n,m}| \Big(\frac{1-\varkappa }{1+\varkappa }\Big) t^{-1} \hat U_2\leq 0
\end{gather}
for all $(X,Y,t)\in D_{W_2}(R_0,t_0)$. Therefore, for all $0<\varepsilon<R_0$ there exist
\begin{gather*}
\delta_\varepsilon =\frac{\varepsilon}{2} \Big(\frac{1-\varkappa }{1+\varkappa }\Big)^{\frac12} t_0^{-\frac{1}{2}-\frac{n-m}{4q}}
\end{gather*}  such that
\begin{gather*}
\sup_{\{ (X,Y): W_2(X,Y)\leq \delta_\varepsilon^2\}} \hat U_2(X,Y,t) \leq  (1+\varkappa ) \delta_\varepsilon^2 <  (1-\varkappa ) t_0^{-1-\frac{n-m}{2q}}\varepsilon^2  \leq
\inf_{\{  (X,Y):  \hat W_2(X,Y,t/t_0)=\varepsilon^2\}} \hat U_2(X,Y,t)
\end{gather*}
for all $1 \leq t/t_0\leq K_\varepsilon$, where $K_\varepsilon=(R_0/\varepsilon)^{4q/(3q-m)}$.
From the last inequalities and estimate \eqref{hU2est} it follows that any solution of system \eqref{XYH} starting in $\{(X,Y): W_2(X,Y)\leq \delta_\varepsilon\}$ at $t=t_0$ satisfies the inequalities:
 $\hat W_2(X(t),Y(t),t/t_0)< \varepsilon^2$ as $t_0\leq t\leq t_0 K_\varepsilon$. Taking into account \eqref{XY2}, we obtain the following:
\begin{gather*}
\omega_2^2 \xi^2(t) t^{\frac{3m-n}{2q}} t_0^{1+\frac{n-m}{2q}} +\eta^2(t)t^{\frac{2m}{q}}<2\varepsilon^2, \quad t_0\leq t\leq t_0 K_\varepsilon.
\end{gather*}
It can easily  be checked that $W_2(\xi(t),\eta(t))<\varepsilon^2$ if $3m-n\geq 0$, and $W_2(\xi(t),\eta(t))<\varepsilon^{4(n+m)/(3n-m)}$ if $3m-n<0$. Thus, the fixed point $(0,0)$ of system \eqref{XEH} and the particular solution $x_\ast(t)$, $y_\ast(t)$ with asymptotics \eqref{AS1} are stable on a finite but asymptotically long time interval.
\end{proof}

{\bf 3}. Finally, consider {\bf Case III}, when $\lambda=0$ and the particular solutions have asymptotics \eqref{AS2} with $x_{n}=\nu_+$.
Define \begin{gather*}
\beta_{n}=\gamma_n(0)+ \delta_{n,q}\frac{5}{2}.
\end{gather*}
Then we have the following.
\begin{Th}\label{st2}
Let $\lambda=0$ and assumptions \eqref{as0}, \eqref{as2} hold.
\begin{itemize}
  \item If $\beta_{n}<0$, then the solution $x_\ast(t)$, $y_\ast(t)$  with asymptotics \eqref{AS2}, $x_{n}=\nu_+$ is stable.
    \item If $\beta_{n}>3\delta_{n,q}$, then the solution $x_\ast(t)$, $y_\ast(t)$  with asymptotics \eqref{AS2}, $x_{n}=\nu_+$ is unstable.
\end{itemize}
\end{Th}
 \begin{proof}
For the solution $x_\ast(t)$, $y_\ast(t)$ with asymptotics \eqref{AS2}, $x_n=\nu_+$, the functions $\mathcal H(\xi,\eta,t)$ and $\mathcal P(\xi,\eta,t)$ in system \eqref{XEH} have the following expansions as $t\to\infty$:
 \begin{gather*}
 \mathcal  H(\xi,\eta,t)=\mathcal H_0(\xi,\eta)+\sum_{k=n}^\infty t^{-\frac{k}{ q}} \mathcal H_k(\xi,\eta) , \quad \mathcal P(\xi,\eta,t)=\sum_{k= n}^\infty t^{-\frac {k}{ q}} \mathcal P_k(\xi,\eta),
\end{gather*}
 where
 \begin{eqnarray*}
  \mathcal  H_0(\xi,\eta)&=&V(\xi;0)-V(0;0)-\xi\partial_x V(0;0)+\frac{\eta^2}{2}=w (0)\frac{\xi^3}{3}+\frac{\eta^2}{2}+  \mathcal O(\rho^4),\\
  \mathcal  H_{n}(\xi,\eta)&=&x_{n}\Big(\partial_x V(\xi;0)-\partial_x V(0;0)-\xi \partial_x^2 V(0;0)\Big) \\
  &&+ \int\limits_0^\eta \Big(F_n(\xi,\theta)-  F_n(0,0)\Big)\,d\theta-\int\limits_0^\xi \Big(G_n(\zeta,0)-  G_n(0,0)\Big)\,d\zeta\\
  &=& \sqrt{\Delta_n}\frac{\xi^2}{2} + \partial_y F_n(0,0)\frac{\eta^2}{2}+\partial_x F_n(0,0)\xi\eta  +\mathcal O(\rho^3),
\end{eqnarray*}
\begin{eqnarray*}
   \mathcal P_{n}(\xi,\eta)&=&\int\limits_0^\eta\Big( \partial_x F_n(\xi,\theta) +\partial_y G_n(\xi,\theta)\Big)\,d\theta=\eta \big(\partial_x F_{n}(0,0)+\partial_y G_{n}(0,0) \big)+\mathcal O(\rho^2)
\end{eqnarray*}
as $\rho=\sqrt{\xi^2+\eta^2}\to 0$.
Note that the function $\mathcal H(\xi,\eta,t)$ is sign indefinite in a neighborhood of the equilibrium $(0,0)$ and can not be used in the construction of a Lyapunov function.
The change of variables:
 \begin{gather}\label{XY3}
    \xi(t)=t^{-\frac{n}{q}}X(t), \quad \eta(t)=t^{-\frac{3n}{2q}}Y(t)
 \end{gather}
transform system \eqref{XEH} into  \eqref{XYH} with
\begin{eqnarray*}
    \hat{ \mathcal H}(X,Y,t) & =& t^{\frac{5n}{2q}}\mathcal H\Big(t^{-\frac{n}{q}}X,t^{-\frac{3n}{2q}}Y,t\Big)+\frac{n}{q}t^{-1}XY\\
        & =& t^{-\frac{n}{2q}}\left( \frac{1}{2}(\omega_3^2 X^2+Y^2)+w(0) \frac{X^3}{3}\right)+\mathcal O\big(R^2t^{-\frac{n}{q}}\big)+\mathcal O (R^2 t^{-1}), \\
    \hat{ \mathcal P}(X,Y,t) & =& t^{\frac{3n}{2q}}\mathcal P\Big(t^{-\frac{n}{q}}X,t^{-\frac{3n}{2q}}Y,t\Big)+\frac{5n}{2q} t^{-1} Y  = t^{-\frac{n}{q}} \big(\beta_{n} Y+\mathcal O(R^2)\big)+\mathcal O\big(R t^{-\frac{n+1}{q}}\big)
\end{eqnarray*}
as $R=\sqrt{X^2+Y^2}\to0$ and $t\to\infty$, where $\omega_3^2=\sqrt{\Delta_n}>0$.
Consider
\begin{gather*}
     U_3(X,Y,t)=t^{\frac{n}{2q}}\hat{\mathcal H}(X,Y,t)- t^{-\frac{n}{2q}}\frac{\beta_{n}}{2}XY
\end{gather*}
as a Lyapunov function candidate. It can easily be checked that
for all $0<\varkappa <1$ there exist $R_1>0$ and $t_{1}>0$ such that
\begin{gather}
\label{W3ineq}
 (1-\varkappa ) W_3(X,Y)  \leq U_3(X,Y,t) \leq  (1+\varkappa ) W_3(X,Y)
\end{gather}
for all $(X,Y,t)\in D_{W_3}(R_1,t_1):=  \{(X,Y,t)\in\mathbb R^3: W_3(X,Y)\leq R_1^2, t\geq t_1\}$, $W_3(X,Y)=(\omega_3^2 X^2+Y^2)/2$. This implies that $U_3(X,Y,t)$ is positive definite. The derivative of $U_3(X,Y,t)$ with respect to $t$ along the trajectories of the system is given by
\begin{gather*}
\frac{dU_3}{dt}\Big|_\eqref{XYH} = \beta_{n} t^{-\frac{n}{q}} W_3(X,Y) \Big(1+\mathcal O(R)+\mathcal O(t^{-\frac 1q})+\mathcal O(t^{-\frac{n}{2q}})\Big), \quad R\to0, \quad t\to\infty.
\end{gather*}
Hence, there exist $0<R_0\leq R_1$ and $t_0\geq t_1$ such that
\begin{eqnarray}
\label{U3ineq}
          \frac{dU_3}{dt}\Big|_\eqref{XYH} \leq -   |\beta_{n}| \Big(\frac{1-\varkappa }{1+\varkappa }\Big) t^{-\frac{n}{q}} U_3\leq 0\quad \text{if}\quad  \beta_n<0,\\
\label{U3est}
          \frac{dU_3}{dt}\Big|_\eqref{XYH} \geq \beta_{n}\Big(\frac{1-\varkappa }{1+\varkappa }\Big) t^{-\frac{n}{q}} U_3\geq 0\quad \text{if}\quad  \beta_n>0
 \end{eqnarray}
 for all $(X,Y,t)\in D_{W_3}(R_0,t_0)$.
Arguing as in the proof of Theorem~\ref{st1}, we see that from \eqref{W3ineq} and \eqref{U3ineq} it follows that the fixed point $(0,0)$ of system \eqref{XYH} and the solution $x_\ast(t)$, $y_\ast(t)$ of system \eqref{FulSys} are stable when $\beta_n<0$.

 Let $\beta_n>3\delta_{n,q}$ and $X(t)$, $Y(t)$ be a solution to system \eqref{XYH} with initial data $X(t_0)$, $Y(t_0)$ such that $W_3(X(t_0),Y(t_0))=\delta_0^2<R_0^2$. Integrating \eqref{U3est} with respect to $t$ and taking into account \eqref{XY3}, \eqref{W3ineq}, we obtain the following estimates:
\begin{gather*}
    \begin{split}
        &  \omega_3^2\xi^2(t)  +\eta^2(t)\geq 2\delta_0^2 \Big(\frac{1-\varkappa }{1+\varkappa }\Big)t^{-\frac{3n}{q}} \exp\left(\frac{q\beta_{n} }{q-n} \Big(\frac{1-\varkappa }{1+\varkappa }\Big)\Big (t^{1-\frac{n}{q}}-t_0^{1-\frac{n}{q}}\Big)\right), \quad  n<q,\\
        &   \omega_3^2\xi^2(t)  +\eta^2(t) \geq 2\delta_0^2 \Big(\frac{1-\varkappa }{1+\varkappa }\Big) t^{-\frac{3n}{q}} \Big(\frac{t}{t_0}\Big)^{\beta_{n}\big(\frac{1-\varkappa }{1+\varkappa }\big) }, \quad n=q
    \end{split}
\end{gather*}
as $t>t_0$. Therefore, choosing $\varkappa\in(0,1)$ small enough ensures the instability of the fixed point $(0,0)$ in system \eqref{XEH} and the solution $x_\ast(t)$, $y_\ast(t)$ in system \eqref{FulSys}: for all $\delta_0>0$ there exists $T_0>t_0$  such that $\xi^2(t)+\eta^2(t)\geq \rho_\ast^2$ as $ t > T_0$ with some $\rho_\ast>0$.
\end{proof}

\begin{Cor}
If $\beta_{n}<0$, the solution $x_\ast(t)$, $y_\ast(t)$  with asymptotics \eqref{AS2}, $x_{n}=\nu_+$ is asymptotically stable.
\end{Cor}

As in the previous case, \eqref{U3est} implies that if $n=q$ and $\beta_{n}>0$, the particular solution $x_\ast(t)$, $y_\ast(t)$ with asymptotics \eqref{AS2}, $x_{n+m}=\nu$ is weakly unstable with the weights $t^{3n/q}$ and $t^{3n/2q}$. Moreover, we have the following:

\begin{Th}\label{mst2}
Let $\lambda=0$, $n=q$  and assumptions \eqref{as0}, \eqref{as1} hold. If $\beta_{n}<1/2$, then the solution $x_\ast(t)$, $y_\ast(t)$  with asymptotics \eqref{AS2}, $x_{n}=\nu_+$ is stable on a finite but asymptotically long time interval.
\end{Th}
\begin{proof}
Consider system \eqref{XEH}, where $x_\ast(t)$, $y_\ast(t)$ is the solution of system \eqref{FulSys} with asymptotics \eqref{AS2}, $x_n=\nu_+$.
It can easily be checked that the change of variables
\begin{gather}\label{XY33}
    \xi(t)=t^{-1}X(t), \quad \eta(t)=t^{-1}Y(t)
\end{gather}
transforms system \eqref{XEH} into \eqref{XYH} with
\begin{eqnarray*}
     \hat{ \mathcal H}(X,Y,t) & = &t^2\mathcal H\Big(t^{-1}X,t^{-1}Y,t\Big)+t^{-1}XY\\
        & =&\frac{Y^2}{2}\big(1+t^{-1}\partial_y F_n(0,0)\big)+ t^{-1}\Big( \omega_3^2\frac{X^2}{2}+w(0) \frac{X^3}{3}+ XY\Big)+\mathcal O (R^2 t^{-1-\frac{1}{q}}), \\
    \hat{ \mathcal P}(X,Y,t) & = &t \mathcal P\Big(t^{-1}X,t^{-1}Y,t\Big)+2 t^{-1} Y
            =  \hat\beta_{n} t^{-1}  Y +\mathcal O\big(R t^{-1-\frac{1}{q}}\big)
 \end{eqnarray*}
as $R=\sqrt{X^2+Y^2}\to0$ and $t\to\infty$, where $\hat\beta_{n}=\beta_n-1/2<0$ and $\omega_3^2=\sqrt{\Delta_n}>0$. As above, we use $\hat{\mathcal H}(X,Y,t)$ as the basis for the Lyapunov function:
\begin{gather*}
\hat U_3(X,Y,t):=  \hat{\mathcal H}(X,Y,t)-t^{-1}\frac{\hat\beta_{n }}{2}XY.
\end{gather*}
It follows easily that $0<\varkappa <1$ there exist $R_1>0$ and $t_1\geq 1$ such that
\begin{gather*}
 (1-\varkappa ) \hat W_3(X,Y,t)  \leq \hat U_3(X,Y,t)\leq (1+\varkappa ) \hat W_3(X,Y,t)
\end{gather*}
for all $(X,Y,t)\in D_{W_3}(R_1,t_1)$, where
\begin{gather*}
    \hat W_3(X,Y,t)=\frac12 \Big( Y^2+ t^{-1}\omega_3^2X^2\Big), \quad W_3(X,Y)\equiv \hat W_3(X,Y,1).
\end{gather*}
Calculating the derivative of the function $\hat U_3(X,Y,t)$ with respect to $t$ along the trajectories of the system, we obtain
\begin{gather*}
\frac{d\hat U_3}{dt}\Big|_\eqref{XYH} =  \frac{\hat\beta_n}{2} t^{-1}   \Big(  Y^2+t^{-1}\big(\omega_2^2X^2+\mathcal O(R^3)\Big)+\mathcal O(t^{-1-\frac{1}{q}} R^2).
\end{gather*}
Hence, for all $0<\varkappa <1$ there exist $0<R_0\leq R_1$ and $t_0\geq t_1$ such that
\begin{gather*}
\frac{d\hat U_3}{dt}\Big|_\eqref{XYH} \leq  -   |\hat\beta_{n}| \Big(\frac{1-\varkappa }{1+\varkappa }\Big) t^{-1} \hat U_3\leq 0
\end{gather*}
for all $(X,Y,t)\in D_{W_3}(R_0,t_0)$.

Arguing as in the previous case, we see that for all $0<\varepsilon<R_0$ there exist $\delta_\varepsilon>0$ such that   any solution of system \eqref{XYH} starting in $\{(X,Y): W_3(X,Y)\leq \delta_\varepsilon\}$ at $t=t_0$ satisfies the inequality:   $\hat W_3(X(t),Y(t),t/t_0)< \varepsilon^2$ as $t_0\leq t\leq t_0 K_\varepsilon$, where $K_\varepsilon=R_0^2 \varepsilon ^{-2}$. Taking into account \eqref{XY33}, we obtain  $W_3(\xi(t),\eta(t))<\varepsilon^2$  as $t_0\leq t\leq t_0 K_\varepsilon$. Thus, the fixed point $(0,0)$ of system \eqref{XEH} and the particular solution $x_\ast(t)$, $y_\ast(t)$ with asymptotics \eqref{AS2} are stable on the asymptotically long time interval.
\end{proof}

The results of this section are shown in Table~\ref{Table2}. Note that the case of stability on a finite but asymptotically long time interval corresponds to metastability of the solutions: the perturbed trajectories spend a long time in the vicinity of the particular solution and ultimately exit from its neighbourhood.

\begin{table}
\begin{tabular}{c|l|c|c|c}
\hline \rule{0cm}{0.5cm}
    {\bf Case}  &   {\bf Assumptions}   &   {\bf Asymptotic }  &   {\bf Stability} &   {\bf Ref. }\\
                        &                               &   {\bf behaviour}  &                          &    \\ [0.1cm]
\hline \rule{0cm}{0.5cm}
    $ \lambda>0$    &   \eqref{as0}  &    \eqref{AS}, $\sigma=-\sqrt\lambda$ &   unstable &    Th.~\ref{uns0} \\ [0.1cm]
\cline{2-5}\rule{0cm}{0.5cm}
                            &  \eqref{as0}, \eqref{as00}, $ \gamma_{n+h}(\lambda)>0$ & \eqref{AS}, $\sigma=\sqrt\lambda$ & unstable & Th.~\ref{st0} \\ [0.1cm]
\cline{2-2}\cline{4-4}\rule{0cm}{0.5cm}
                            &  \eqref{as0}, \eqref{as00}, $ \gamma_{n+h}(\lambda)<0$ &   & stable & \\ [0.1cm]
\cline{2-2}\cline{4-5}\rule{0cm}{0.5cm}
                            &  \eqref{as0}, \eqref{as00H}, $ d_{n}<0$ & & unstable & Th.~\ref{st0H} \\ [0.1cm]
\cline{2-2}\cline{4-4}\rule{0cm}{0.5cm}
                            &  \eqref{as0}, \eqref{as00H}, $ d_{n}>0$ &   & stable & \\ [0.1cm]
\hline\hline \rule{0cm}{0.5cm}
    $ \lambda=0$    &   \eqref{as0}, \eqref{as1}  &    \eqref{AS1}, $x_{n+m}=-\mu$ &   unstable &    Th.~\ref{uns1} \\ [0.1cm]
\cline{2-5}\rule{0cm}{0.5cm}
                            &  \eqref{as0}, \eqref{as1}, $  \gamma_n(0)>\delta_{n,q}\frac{(q+m)}{4q}$ & \eqref{AS1}, $x_{n+m}=\mu$ &unstable & Th.~\ref{st1} \\ [0.1cm]
\cline{2-2}\cline{4-4}\rule{0cm}{0.5cm}
                            &  \eqref{as0}, \eqref{as1}, $ \gamma_n(0)<-\delta_{n,q}\frac{5(q+m)}{4q}$ &   & stable & \\ [0.1cm]
\cline{2-2}\cline{4-5}\rule{0cm}{0.5cm}
                            &  \eqref{as0}, \eqref{as1}, $n=q$, $-\frac{5(q+m)}{4q}< \gamma_n(0)<-\frac{q+3m}{2q}$ &   & metastable & Th.~\ref{mst1} \\ [0.1cm]
\hline \hline \rule{0cm}{0.5cm}
    $ \lambda=0$    &   \eqref{as0}, \eqref{as2}  &     \eqref{AS2}, $x_{n}=\nu_-$ &   unstable &    Th.~\ref{uns2}   \\ [0.1cm]
\cline{2-5}\rule{0cm}{0.5cm}
                            &  \eqref{as0}, \eqref{as2}, $\gamma_n(0)> \delta_{n,q}\frac12$ &   \eqref{AS2},  $x_{n}=\nu_+$   & unstable & Th.~\ref{st2} \\ [0.1cm]
\cline{2-2}\cline{4-4}\rule{0cm}{0.5cm}
                            &  \eqref{as0}, \eqref{as2}, $\gamma_n(0)<-\delta_{n,q}\frac 52$ &     &  stable &   \\ [0.1cm]
\cline{2-2}\cline{4-5}\rule{0cm}{0.5cm}
                            &  \eqref{as0}, \eqref{as2}, $n=q$,  $-\frac 52<\gamma_n(0)<-2$ &      &  metastable & Th.~\ref{mst2} \\ [0.1cm]
\hline
\end{tabular}
\bigskip
\caption{{\small Stability of the particular solutions $x_\ast(t)$, $y_\ast(t)$ of system \eqref{FulSys}.}}
\label{Table2}
\end{table}

\section{Examples}
\label{sec4}
{\bf 1.} Consider again equation \eqref{Example} that corresponds to system \eqref{FulSys} with $w(x)\equiv 1$, $F(x,y,t)\equiv 0$, $G(x,y,t)\equiv t^{-\kappa} (By+C)$. It can easily be checked that this system satisfies \eqref{as0} and \eqref{as00} with $\kappa=n/q$, $h=0$, $\gamma_{n}(\lambda)=B$. Consequently, if $\lambda>0$, there are two particular solutions with asymptotics \eqref{AS}: $x_\pm(t)=\pm\sqrt\lambda + \mathcal O(t^{-\kappa})$, $y_\pm(t)=\mathcal O(t^{-1-\kappa})$ as $t\to \infty$. The solution $x_-(t)$, $y_-(t)$ is unstable (Theorem~\ref{uns0}); the solution $x_+(t)$, $y_+(t)$ is unstable if $B>0$ and stable if $B<0$  (Theorem~\ref{st0}). It follows from Corollary~\ref{Cor1} that $x_+(t)$, $y_+(t)$ is asymptotically stable if $B<0$ and $0<\kappa\leq 1$. If $B=0$, equation \eqref{Example} satisfies \eqref{as00H} with $d_n=C$. In this case, from Theorem~\ref{st0H} it follows that the solution $x_+(t)$, $y_+(t)$ is (neutrally) stable if $C>0$, and unstable if $C<0$ (see~Fig.~\ref{f1}, b).

Note that system \eqref{Example} also satisfies \eqref{as1} with $m=0$, $G_{n+m}(0,0)=C$ and $\alpha_{n,m}=B+\delta_{\kappa,1}(5/4)$. Therefore, if $\lambda=0$ and $C>0$, there are two particular solutions with asymptotics \eqref{AS1}: $\tilde x_\pm(t)=\pm \mu t^{-\kappa/2}(1+o(1))$, $\tilde y_\pm(t)=\mathcal O(t^{-1-\kappa/2})$ as $t\to \infty$ with $\mu=\sqrt{C}$. The solution $\tilde x_-(t)$, $\tilde y_-(t)$ is unstable (Theorem~\ref{uns0}); the solution $\tilde x_+(t)$, $\tilde y_+(t)$ is unstable if $B>\delta_{\kappa,1}(1/4)$ and stable if $B<-\delta_{\kappa,1}(5/4)$  (Theorem~\ref{st1}). From Theorem~\ref{mst1} it follows that if $\kappa=1$ $(n=q)$ and $-5/4<B<-1/2$, then the solution $\tilde x_+(t)$, $\tilde y_+(t)$ is metastable (see~Fig.~\ref{f2xy}).

\begin{figure}
\centering
\subfigure[$B=-1.5$]{\includegraphics[width=0.3\linewidth]{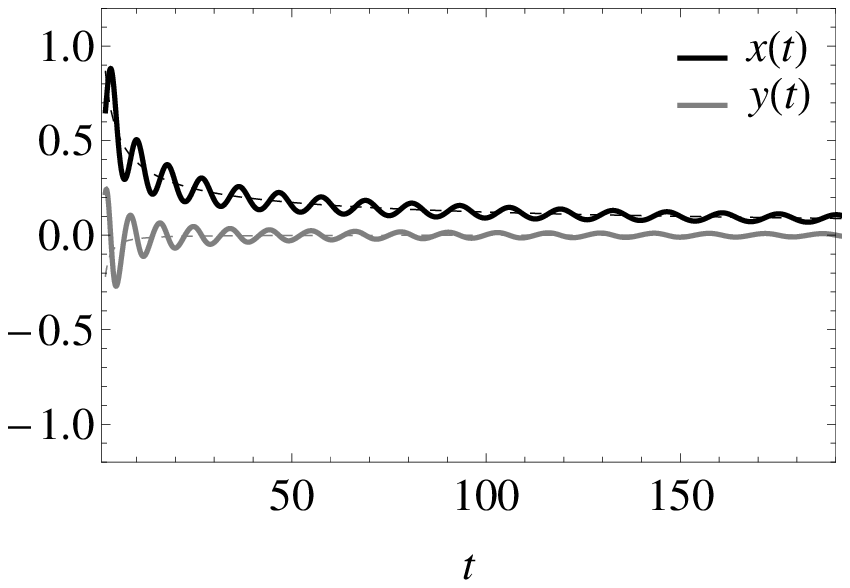}}
\subfigure[$B=-0.6$]{\includegraphics[width=0.3\linewidth]{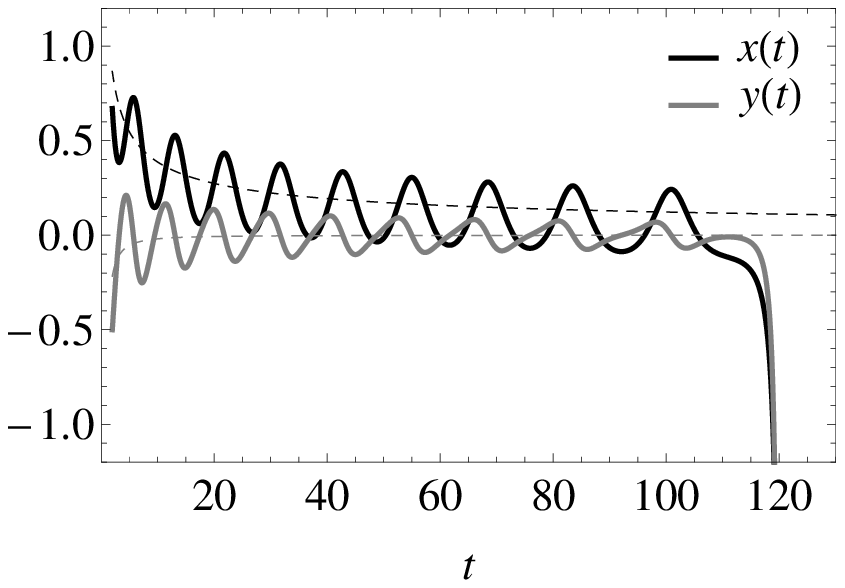}}
\subfigure[$B=0.5$]{\includegraphics[width=0.3\linewidth]{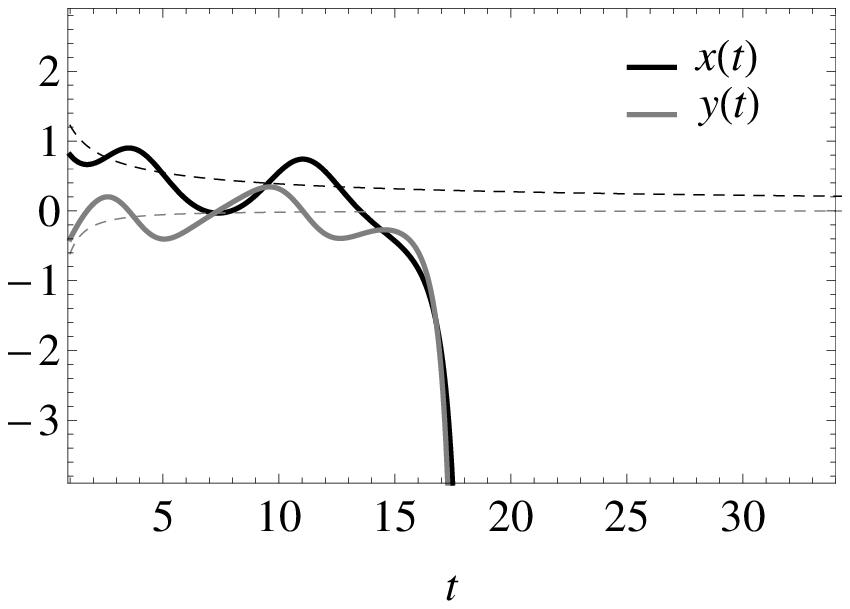}}
\caption{\small The evolution of $x(t)$ and $y(t)$ for solutions of equation \eqref{Example} with  $\lambda=0$, $C=1.5$, $\kappa=1$.} \label{f2xy}
\end{figure}

{\bf 2.} Consider
\begin{gather}\label{Example2}
\frac{dx}{dt}=y, \quad \frac{dy}{dt}=-(x^2-\lambda)(1+x)+t^{-1}\big(Ax+By+g(x,y)\big)+t^{-2} C,
\end{gather}
where $g(x,y)=x^2+y^2+xy$, $\lambda,A,B,C={\hbox{\rm const}}$, $C\neq 0$, $\lambda<1$. It can easily be checked that this system has the form \eqref{FulSys} with $w(x)=1+x$, $q=1$, $F(x,y,t)\equiv 0$, $G(x,y,t)\equiv t^{-1}G_1(x,y)+t^{-2} G_2(x,y)$, $G_1(x,y)\equiv Ax+By+g(x,y)$, $G_2(x,y)\equiv C$. This system satisfies \eqref{as0} and \eqref{as00} with $n=1$, $h=0$, $\gamma_{n}(\lambda)=B+\sqrt\lambda$. Hence, if $0<\lambda<1$, there are two particular solutions with asymptotics \eqref{AS}: $x_\pm(t)=\pm\sqrt\lambda + \mathcal O(t^{-1})$, $y_\pm(t)=\mathcal O(t^{-2})$ as $t\to \infty$. In this case, from Theorem~\ref{uns0} it follows that the solution $x_-(t)$, $y_-(t)$ is unstable. From Theorem~\ref{st0} and Corollary~\ref{Cor1} it follows that the solution $x_+(t)$, $y_+(t)$ is unstable if $B>-\sqrt\lambda$, and asymptotically stable if $B<-\sqrt\lambda$ (see~Fig.~\ref{f31xy}).

\begin{figure}
\centering
\subfigure[$B=-1$]{\includegraphics[width=0.3\linewidth]{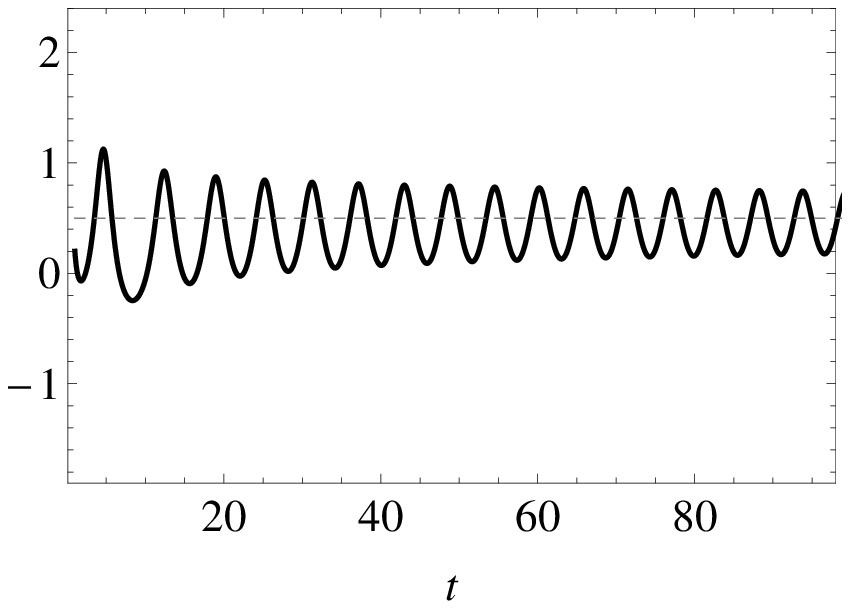}}
\subfigure[$B=0.5$]{\includegraphics[width=0.3\linewidth]{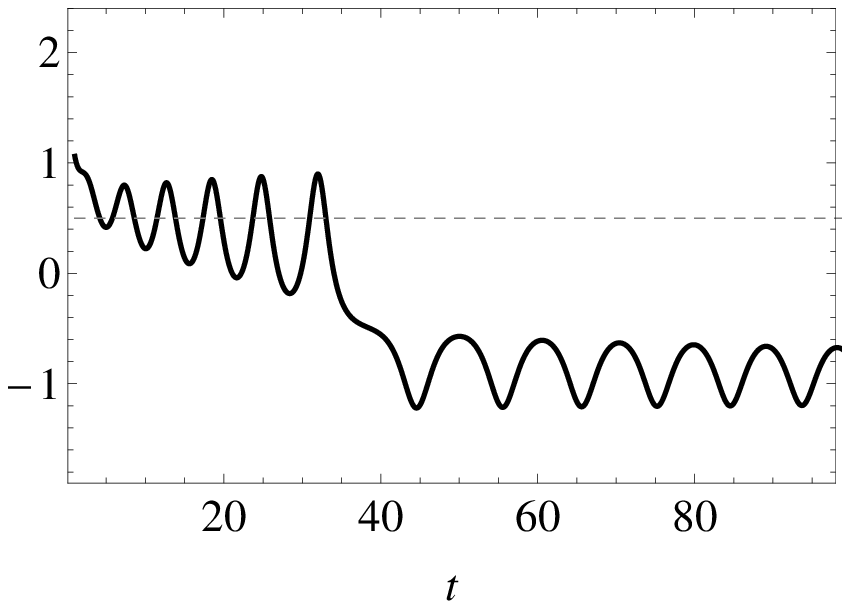}}
\caption{\small The evolution of $x(t)$ for solutions of equation \eqref{Example2} with  $\lambda=0.25$, $A=C=1$.} \label{f31xy}
\end{figure}

It is readily seen that system \eqref{Example2} satisfies additionally \eqref{as2} with $G_{2n}(0,0)=C$, $\Delta_n=A^2+4C$ and $\beta_{n}=B+5/2$.
Therefore, if $\lambda=0$ and $A^2+4C>0$, system \eqref{Example2} has two particular solutions with asymptotics \eqref{AS2}: $\tilde x_\pm(t)=\nu_\pm t^{-1}(1+o(1))$, $\tilde y_\pm(t)=\mathcal O(t^{-2})$ as $t\to \infty$ with $\nu_\pm=(A\pm\sqrt{A^2+4C})/2$. From Theorem~\ref{uns2} we know that the solution $\tilde x_-(t)$, $\tilde y_-(t)$ is unstable.  By applying Theorem~\ref{st2}, we see that the solution $\tilde x_+(t)$, $\tilde y_+(t)$ is unstable if $B>1/2$ and stable if $B<-5/2$ (see Fig.~\ref{f3} and~\ref{f3xy}, a, c). From Theorem~\ref{mst1} it follows that if $-5/2<B<-2$, the solution $\tilde x_+(t)$, $\tilde y_+(t)$ is metastable (see~Fig.~\ref{f3} and~\ref{f3xy}, b).

\begin{figure}
\centering
\subfigure[$B=-3$ ]{\includegraphics[width=0.3\linewidth]{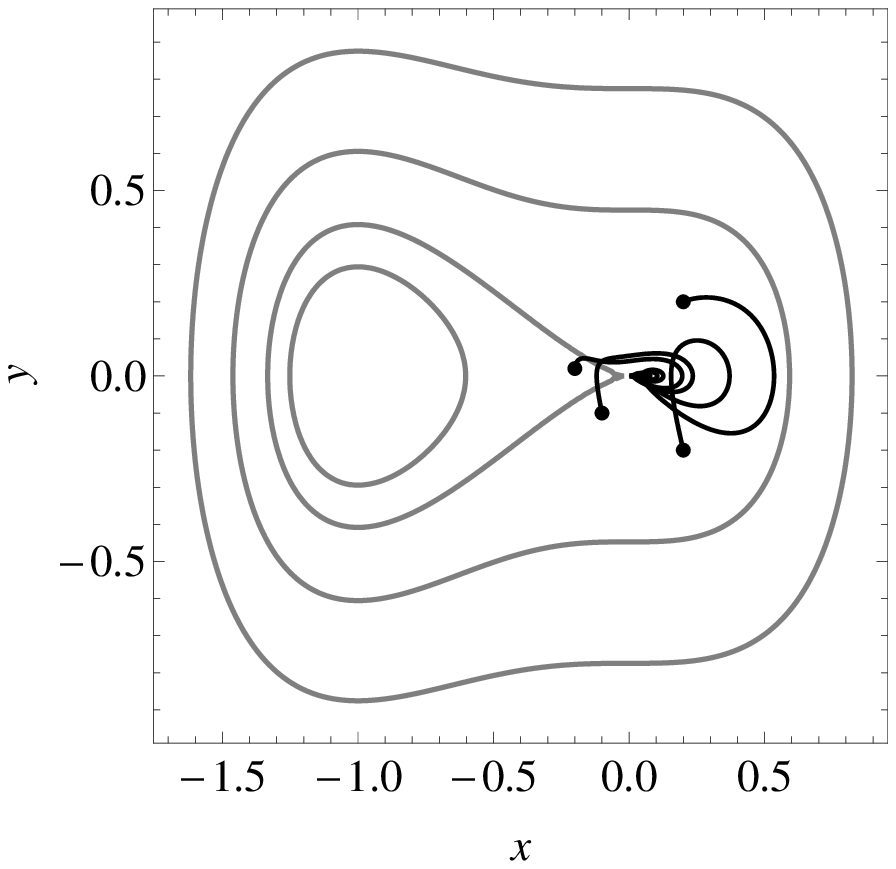}}
\subfigure[ $B=-2.25$]{\includegraphics[width=0.3\linewidth]{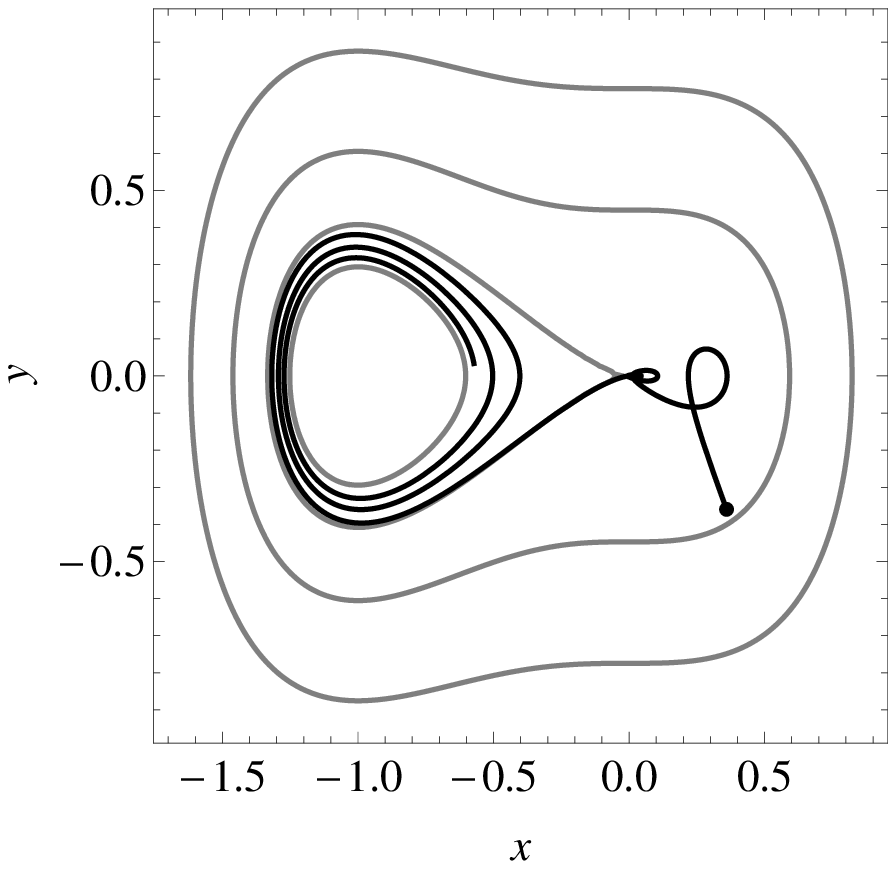}}
\subfigure[$B=1$]{\includegraphics[width=0.3\linewidth]{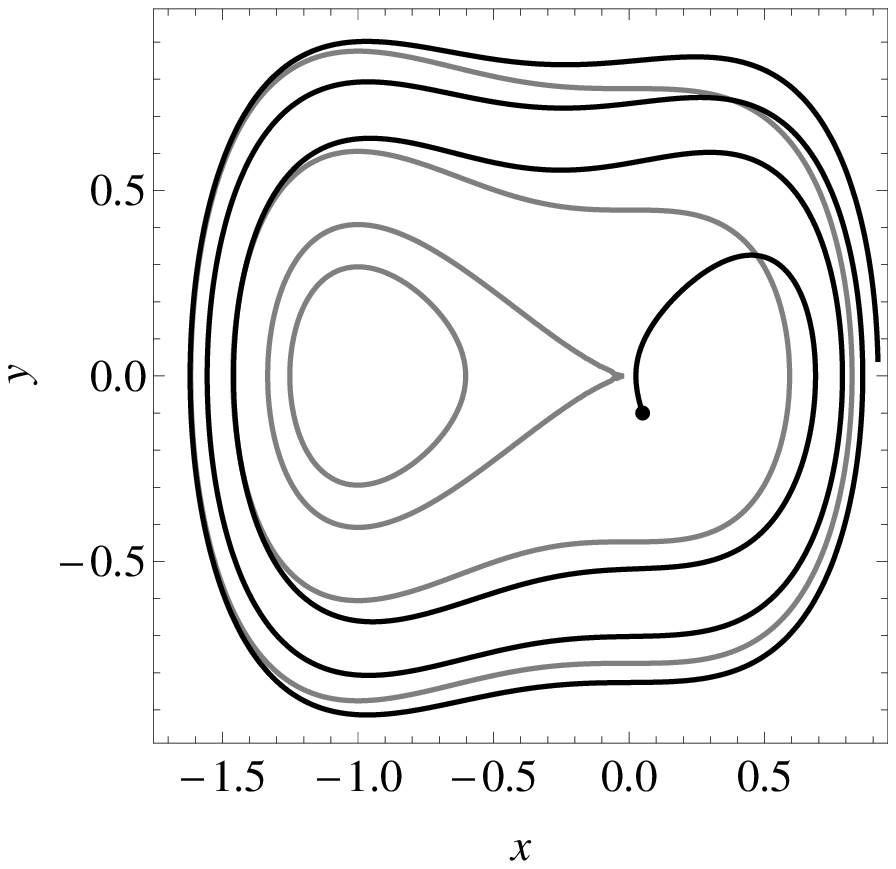}}
\caption{\small The evolution of $(x(t),y(t))$ for solutions of \eqref{Example2} with $\lambda=0$, $A=C=1$. The black points correspond to initial data $(x(2),y(2))$. The gray curves correspond to level lines of $H(x,y;0)$.} \label{f3}
\end{figure}

\begin{figure}
\centering
\subfigure[$B=-3$]{\includegraphics[width=0.3\linewidth]{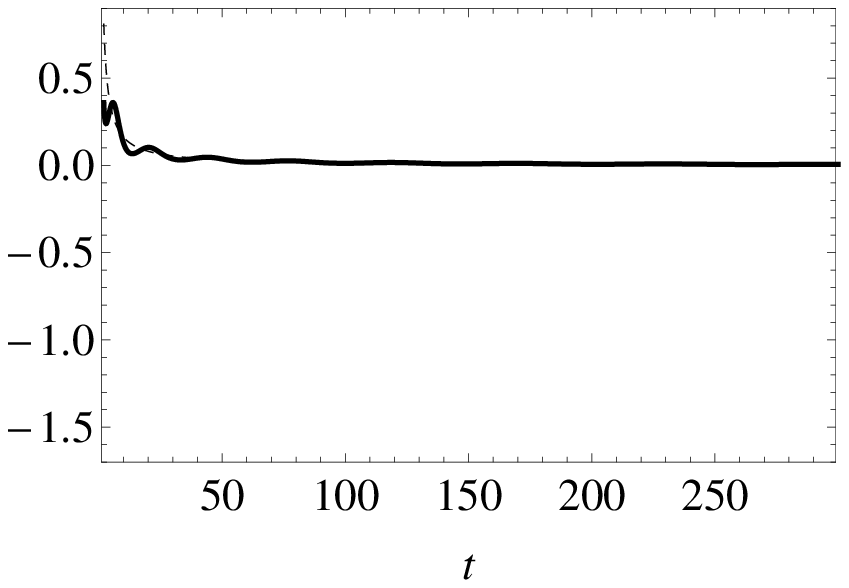}}
\subfigure[$B=-2.25$]{\includegraphics[width=0.3\linewidth]{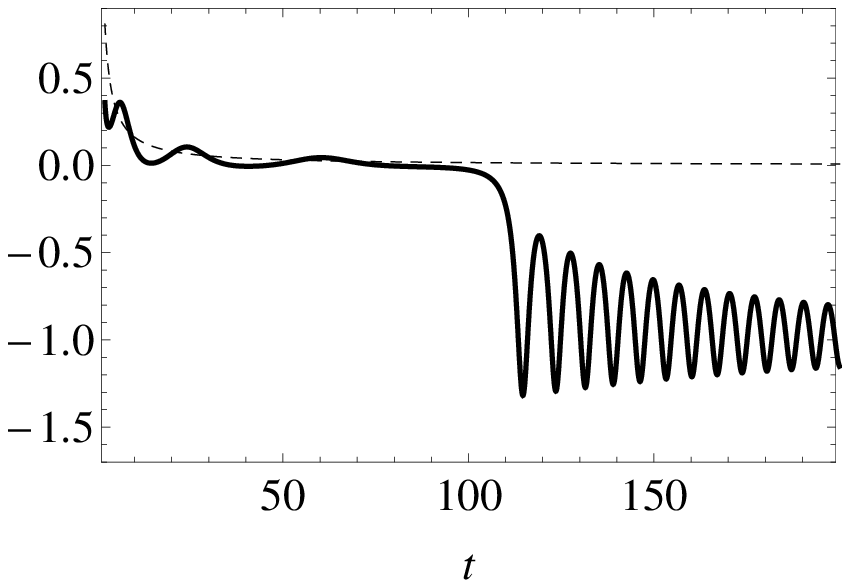}}
\subfigure[$B=1$]{\includegraphics[width=0.3\linewidth]{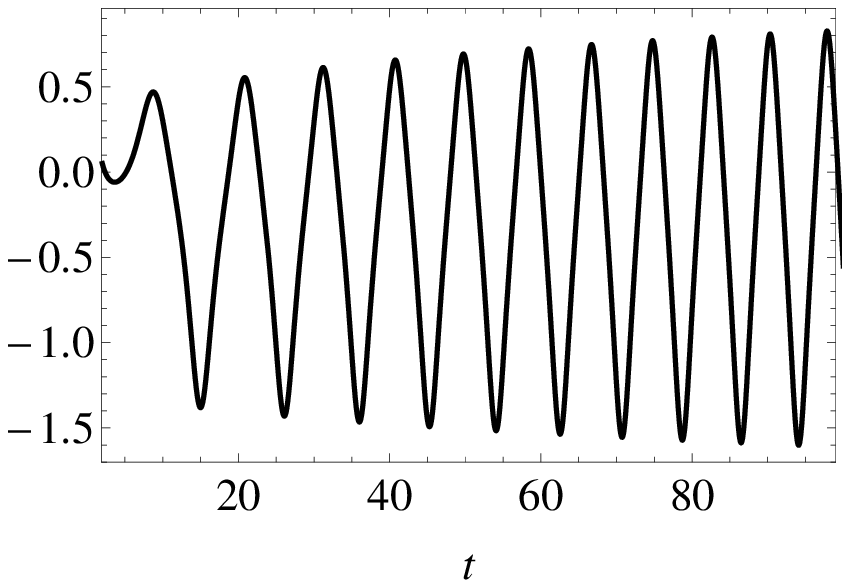}}
\caption{\small The evolution of $x(t)$ for solutions of equation \eqref{Example2} with  $\lambda=0$, $A=C=1$.} \label{f3xy}
\end{figure}

\section{Conclusion}
Thus, the influence of damped perturbations on autonomous systems with a centre-saddle bifurcation has been investigated.
We have shown that depending on the structure and the parameters of disturbances the qualitative behaviour of perturbed systems can be quite different from that of the corresponding limiting systems. Specifically, if $\lambda>0$, there are two particular solutions tending to fixed points of the limiting system.
The solution corresponding to the saddle is unstable regardless of decaying perturbations, while the other solution can be asymptotically stable, neutrally stable or unstable. In case of stability, there exists a family of solutions to the perturbed system with similar long-term behaviour.

When the parameter $\lambda$ passes through the bifurcation value, the centre and saddle coalesce and disappear in the limiting system. We have shown that the decaying perturbations can break such transition. In particular, if $\lambda=0$, the perturbed system can have a pair of different particular solutions tending to a degenerate fixed point of the limiting system. These solutions are a saddle and a centre in the asymptotic limit. The saddle-type solution is always unstable, while the second solution, depending on the perturbations, can be stable, metastable or unstable. In case of metastability,  the trajectories of the perturbed system remain in the vicinity of the particular solution for a sufficiently long time interval but eventually leave its neighbourhood. We have also described the conditions under which such solutions do not appear, and the centre-saddle bifurcation is preserved in the perturbed system.

\section*{Acknowledgments}
The research is supported by the Russian Science Foundation (Grant No. 20-11-19995).

}
\end{document}